\documentclass[11pt]{amsart}
\usepackage{latexsym}
\usepackage{amsmath,amssymb,amsfonts}
\usepackage{graphics}
\usepackage[all]{xy}

\def\Bbb{\mathbb}

\def\eea{\end{eqnarray*}}

\newtheorem{defn}{Definition}
\newtheorem{thm}{Theorem}[section]

\newtheorem{cor}[thm]{Corollary}
\newtheorem{lem}[thm]{Lemma}

\newenvironment{rmk}{\mbox{ }\\{\bf  Remark}\mbox{ }}{
\hfill $\Box$\mbox{}\bigskip}

\begin{document}
\renewcommand{\theequation}{\thesection.\arabic{equation}}

\title[Finite group actions and $G$-monopole classes]{Finite group actions and $G$-monopole classes on smooth 4-manifolds}

\author{Chanyoung Sung}
\date{\today}

\address{Dept. of mathematics education \\
Korea national university of education}
\email{cysung@kias.re.kr}
\thanks{This work was supported by the National Research Foundation of Korea grant funded by the Korea government. (NRF-2014R1A1A4A01008987)}

\begin{abstract}
On a smooth closed oriented $4$-manifold $M$ with a smooth action
by a compact Lie group $G$,  we define a $G$-monopole class as an
element of $H^2(M;\Bbb Z)$ which is the first Chern class of a
$G$-equivariant Spin$^c$ structure which has a solution of the
Seiberg-Witten equations for any $G$-invariant Riemannian metric
on $M$.

We find $\Bbb Z_k$-monopole classes on some $\Bbb Z_k$-manifolds such as the connected sum of $k$ copies of a 4-manifold with nontrivial mod 2 Seiberg-Witten
invariant or Bauer-Furuta invariant, where the $\Bbb Z_k$-action
is a cyclic permutation of $k$ summands.

As an application, we produce infinitely many exotic non-free actions of $\Bbb
Z_k\oplus H$ on some connected sums of finite number of
$S^2\times S^2$, $\Bbb CP_2$, $\overline{\Bbb CP}_2$, and $K3$
surfaces, where $k\geq 2$, and $H$ is any nontrivial finite group acting freely on $S^3$.
\end{abstract}
\maketitle
\setcounter{section}{0}
\setcounter{equation}{0}

\section{Introduction}
Let $M$ be a smooth closed oriented manifold of dimension 4.
A second cohomology class of $M$ is called a \emph{monopole class}
if it arises as the first Chern class of a Spin$^c$ structure for
which the Seiberg-Witten equations
$$\left\{
\begin{array}{ll} D_A\Phi=0\\
  F_{A}^+=\Phi\otimes\Phi^*-\frac{|\Phi|^2}{2}\textrm{Id},
\end{array}\right.
$$
admit a solution for every choice of a Riemannian metric. Clearly
a basic class, i.e. the first Chern class of a Spin$^c$ structure with a nonzero Seiberg-Witten invariant is a monopole class. However, ordinary Seiberg-Witten invariants which are gotten by the intersection
theory on the moduli space of solutions $(A,\Phi)$ of the above equations is trivial in many
important cases, for example connected sums of 4-manifolds with $b_2^+>0$.

Bauer and Furuta \cite{BF, bau} made a breakthrough in detecting a
monopole class on connected sums of 4-manifolds. Their 
brilliant idea is to generalize the Pontryagin-Thom construction
to the ``monopole map", i.e. a proper map between infinite-dimensional spaces given by the Seiberg-Witten equations,
and take some sort of a stably-framed bordism class of the Seiberg-Witten moduli space as a new 
invariant. However its applications are still limited in that
this new invariant which is expressed as a stable cohomotopy class
is difficult to compute (especially when $b_1\ne 0$), and we are seeking after further refined
invariants of the Seiberg-Witten moduli space.

In the meantime, sometimes we need a solution of the
Seiberg-Witten equations for a specific metric rather than any
Riemannian metric. The case we have in mind is the one when a
manifold $M$ and its Spin$^c$ structure $\frak{s}$ admit a smooth
action by a compact Lie group $G$ and we are concerned with finding
a solution of the Seiberg-Witten equations for any $G$-invariant
metric.

Thus for a $G$-invariant metric on $M$ and a $G$-invariant perturbation of the Seiberg-Witten equations, we consider the \emph{$G$-monopole moduli space} $\frak{X}$ consisting of their $G$-invariant solutions
modulo gauge equivalence. One can easily see that the ordinary
moduli space $\frak{M}$ is acted on by $G$, and $\frak{X}$ turns out to be
a subset of its $G$-fixed point set. The intersection theory on $\frak{X}$ will
give the \emph{$G$-monopole invariant} $SW^{G}_{M,\frak{s}}$
defined first by Y. Ruan \cite{ruan}, which is expected to be
different from the ordinary Seiberg-Witten invariant $SW_{M,\frak{s}}$.

In view of this, the following definition is relevant.
\begin{defn}
Suppose that $M$ admits a smooth action by a compact Lie group $G$
preserving the orientation of $M$.

A second cohomology class of $M$ is called a $G$-\emph{monopole
class} if it arises as the first Chern class of a $G$-equivariant
Spin$^c$ structure for which the Seiberg-Witten equations admit a
$G$-invariant solution for every $G$-invariant Riemannian metric of $M$.
\end{defn}
When a $G$-monopole invariant is nonzero, its first Chern class has
to be a $G$-monopole class. As explain in \cite{sung3}, the cases we are aiming at are those for finite $G$.  If a compact connected Lie group $G$  has positive dimension and is not a torus $T^n$, then $G$ contains a Lie subgroup isomorphic to $S^3$ or $S^3/\Bbb Z_2$, and hence $M$ admitting an effective action of such $G$ must have a $G$-invariant metric of positive scalar curvature by the well-known Lawson-Yau theorem \cite{law-yau} so that $M$ with $b_2^+(M)^G>1$ has no $G$-monopole class, where $b_2^+(M)^G$ denotes the maximal dimension of $G$-invariant positive-definite subspaces of $H^2(M;\Bbb R)$ under the intersection pairing. On the other hand, Seiberg-Witten invariants of a 4-manifold with an effective $S^1$ action were extensively studied by S. Baldridge \cite{bal1, bal2, bal3}.

As another invariant for detecting a $G$-monopole class, one can also generalize the Bauer-Furuta invariant $BF_{M,\frak{s}}$ of \cite{BF, bau} to $\overline{BF}^{G}_{M,\frak{s}}$, which is
roughly the $S^1$-equivariant stable cohomotopy class of the monopole map between
$G$-invariant subsets of  the associated Hilbert bundles over $T^{b_1(M)^G}$ where $b_1(M)^G$ is the dimension of the space of $G$-invariant elements in $H^1(M;\Bbb R)$. As will be discussed later, this is a little different from the $G$-equivariant Bauer-Furuta invariant $BF^G_{M,\frak{s}}$ of \cite{szymik,naka1}, which is the $(G\times S^1)$-equivariant stable cohomotopy class of the monopole map between the original Hilbert bundles over the Picard torus $T^{b_1(M)}$.

Using these invariants, we find $G$-monopole classes in some
connected sums which have vanishing Seiberg-Witten invariants :
\begin{thm}\label{firstth}
Let $M$ and $N$ be  smooth closed oriented connected 4-manifolds satisfying
$b_2^+(M)> 1$ and $b_2^+(N)=0$, and $\bar{M}_k$ for any $k\geq 2$ be the connected sum
$M\#\cdots \#M\# N$ where there are $k$ summands of $M$.

Suppose that $\Bbb Z_k$ acts effectively on $N$ in a smooth
orientation-preserving way such that it is free or has at least one fixed point,
and that $N$ admits a Riemannian metric of positive scalar curvature invariant under the $\Bbb Z_k$-action and a $\Bbb Z_k$-equivariant Spin$^c$ structure $\frak{s}_N$ with $c_1^2(\frak{s}_N)=-b_2(N)$.

Define a $\Bbb Z_k$-action on $\bar{M}_k$ induced from that of $N$
and the cyclic permutation of  the $k$ summands of $M$ glued along
a free orbit in $N$, and let $\bar{\frak{s}}$ be the
Spin$^c$ structure on $\bar{M}_k$ obtained by gluing $\frak{s}_N$ and a Spin$^c$
structure $\frak{s}$ of $M$.

Then for any $\Bbb Z_k$-action on $\bar{\frak{s}}$ covering the above
$\Bbb Z_k$-action on $\bar{M}_k$, $SW^{\Bbb
Z_k}_{\bar{M}_k,\bar{\frak{s}}}$ mod 2 is nontrivial if
$SW_{M,\frak{s}}$ mod 2 is nontrivial, and also $BF^{\Bbb
Z_k}_{\bar{M}_k,\bar{\frak{s}}}$ is nontrivial, if
$BF_{M,\frak{s}}$ is nontrivial.
\end{thm}

As an application of a $G$-monopole class to differential
topology, we can detect exotic smooth group actions on some smooth
4-manifolds, by which we mean topologically equivalent but
smoothly inequivalent actions. We say that two smooth group actions $G_1$ and $G_2$ on a smooth manifold $M$ is $C^m$-equivalent for $m=0,1,\cdots,\infty$, if there exists a $C^m$-homeomorphism $f:M\rightarrow M$ such that $$G_1=f\circ G_2\circ f^{-1}.$$
Such exotic smooth actions of finite groups on smooth 4-manifolds have been found abundantly, for eg, \cite{FS1, CS, Go1, Ak, Go2, Ue, FS2}. But all of
them are either free or cyclic actions.

In the final section, we use $G$-monopole invariants
to find infinitely many non-free  non-cyclic exotic group actions.
 For example, for $k\geq 2$ and any nontrivial finite group $H$ acting freely on $S^3$, there exist infinitely many exotic non-free actions of $\Bbb Z_k\oplus H$  on some connected sums of finite numbers of $S^2\times S^2$, $\Bbb CP_2$, $\overline{\Bbb CP}_2$, and $K3$ surfaces.

The above theorem may be generalized to other types of $N$, and we leave a more complete study to a future project. Applications to Riemannian geometry such as $G$-invariant Einstein metrics and $G$-Yamabe invariant are rather straightforward from the curvature estimates, and are dealt with in \cite{sung3}.

\section{$G$-monopole invariant}
Let $M$ be  a  smooth closed oriented 4-manifold.  Suppose that a
compact Lie group $G$ acts on $M$ smoothly preserving the
orientation, and this action lifts to an action on a Spin$^c$
structure $\mathfrak{s}$ of $M$. Once there is a lifting, any
other lifting differs from it by an element of $Map(G\times M,
S^1)$. We fix a lifting and put a $G$-invariant Riemannian metric
$g$ on $M$. Then the associated spinor bundles $W_{\pm}$ are also
$G$-equivariant, and we let $\Gamma(W_{\pm})^G$ be the set of its
$G$-invariant sections. When we put $G$ as a superscript on the right of a set,
we always mean the subset consisting of its $G$-invariant
elements. Thus $\mathcal{A}(W_+)^G$ is the space of $G$-invariant
connections on $\det (W_+)$, which is identified as the space $\Gamma(\Lambda^1(M;i\Bbb R))^G$ of
$G$-invariant purely-imaginary valued 1-forms on $M$, and $\mathcal{G}^G=Map(M,S^1)^G$ is the group of $G$-invariant gauge transformations.

The perturbed Seiberg-Witten equations give a smooth map $$H:
\mathcal{A}(W_+)^G\times \Gamma(W_+)^G\times \Gamma(\Lambda^2_+(M;i\Bbb
R))^G\rightarrow \Gamma(W_-)^G\times \Gamma(\Lambda^2_+(M;i\Bbb R))^G$$ defined by
$$H(A,\Phi,\varepsilon)=(D_A\Phi,
F_A^+-\Phi\otimes\Phi^*+\frac{|\Phi|^2}{2}\textrm{Id}+\varepsilon),$$
where the domain and the range are endowed with $L_{l+1}^2$ and
$L_l^2$  Sobolev norms for a positive integer $l$ respectively,
and $D_A$ is a Spin$^c$ Dirac operator. The $G$-monopole moduli
space $\frak{X}_\varepsilon$ for a perturbation $\varepsilon$ is
then defined as
$$\frak{X}_\varepsilon:=H^{-1}_\varepsilon(0)/ \mathcal{G}^G$$ where $H_\varepsilon$ denotes
$H$ restricted to $\mathcal{A}(W_+)^G\times \Gamma(W_+)^G\times\{\varepsilon\}$.

In the followings, we give a detailed proof that $\frak{X}_\varepsilon$ for generic $\varepsilon$ and finite $G$ is a smooth compact manifold, because some statements in \cite{ruan, cho}
are incorrect or without proofs.
\begin{lem}\label{saveyou}
The quotient map $$p : (\mathcal{A}(W_+)^G\times (\Gamma(W_+)^G-\{0\}))/ \mathcal{G}^G\rightarrow (\mathcal{A}(W_+)^G\times (\Gamma(W_+)^G-\{0\}))/
\mathcal{G}$$ is bijective, and hence $\mathfrak{X}_\varepsilon$ is a subset of the
ordinary Seiberg-Witten moduli space $\frak{M}_\varepsilon$.
\end{lem}
\begin{proof}
Obviously $p$ is surjective, and to show that $p$ is injective, suppose that  $(A_1,\Phi_1)$ and
$(A_2,\Phi_2)$ in $\mathcal{A}(W_+)^G\times (\Gamma(W_+)^G-\{0\})$ are
equivalent under $\gamma\in\mathcal{G}$. Then $$A_1=A_2-2d\ln \gamma, \ \ \ \textrm{and}\ \ \
\Phi_1=\gamma\Phi_2.$$ By the first equality, $d\ln \gamma$ is $G$-invariant.

Let $S$ be the subset of $M$ where $\gamma$ is $G$-invariant. By the continuity of $\gamma$, $S$ must be a closed subset. Since $S$ contains a nonempty subset $$\{x\in M| \Phi_1(x)\ne 0\},$$ $S$ is nonempty.
It suffices to show that ${S}$ is open. Let $x_0\in {S}$. Then we have that for any $g\in G$, $$g^*\ln\gamma(x_0)=\ln\gamma(x_0), \ \ \ \textrm{and}\ \ \ g^*d\ln\gamma=d\ln\gamma,$$ which implies that $g^*\ln\gamma=\ln\gamma$ on an open neighborhood of $x_0$ on which $g^*\ln\gamma$ and $\ln\gamma$ are well-defined. By the compactness of $G$, there exists an open neighborhood of $x_0$ on which $g^*\ln\gamma$ is well-defined for all $g\in G$, and $\ln \gamma$ is $G$-invariant. This proves the openness of $S$.
\end{proof}

As in the ordinary Seiberg-Witten moduli space, the transversality is
obtained by a generic perturbation $\varepsilon$ :
\begin{lem}
$H$ is a submersion at each $(A,\Phi,\varepsilon)\in H^{-1}(0)$ for nonzero
$\Phi$.
\end{lem}
\begin{proof}
Obviously $d{H}$ restricted to the last factor of the domain is onto
the last factor of the range. Using the surjectivity in the ordinary
setting,  for any element $\psi\in \Gamma(W_-)^G$, there exists
an element  $(a,\varphi)\in  \mathcal{A}(W_+)\times \Gamma(W_+)$
such that $d{H}(a,\varphi,0)=\psi$. The average
$$(\tilde{a},\tilde{\varphi}):=\int_G h^*(a,\varphi)\
d\mu(h):=( \int_G h^*a\ d\mu(h) , \int_G h^*\varphi\ d\mu(h))$$
using a unit-volume $G$-invariant metric on $G$ is an element of
$\mathcal{A}(W_+)^G\times \Gamma(W_+)^G$. It follows from the
smoothness of the $G$-action that every $h^*(a,\varphi)$ and hence
$(\tilde{a},\tilde{\varphi})$ belong to the same Sobolev space as
$(a,\varphi)$.
Moreover it still satisfies
\begin{eqnarray*}
d{H}(\tilde{a},\tilde{\varphi},0)&=& \int_G dH (h^*(a,\varphi,0))\
d\mu(h)\\ &=& \int_G h^* dH ((a,\varphi,0))\
d\mu(h)\\ &=& \int_G h^* \psi\
d\mu(h)\\ &=& \psi,
\end{eqnarray*}
where we used the fact that $d{H}$ is a $G$-equivariant differential operator. This completes the proof.
\end{proof}

Assuming that $b_2^{+}(M)^G$ is nonzero,
$\mathfrak{X}_\varepsilon$ consists of irreducible solutions. By
the above lemma, $\cup_{\varepsilon}\mathfrak{X}_\varepsilon$ is a
smooth submanifold, and applying Smale-Sard theorem to the
projection map onto $\Gamma(\Lambda^2_+(M;i\Bbb R))^G$,
$\mathfrak{X}_\varepsilon$  for generic $\varepsilon$ is also
smooth. (Nevertheless $\frak{M}_\varepsilon$ for that $\varepsilon$ may not be smooth in general. Its obstruction is explained in \cite{cho}.)   From now on, we will always assume that a generic
$\varepsilon$ is chosen so that $\frak{X}_\varepsilon$ is smooth,
and often omit the notation of $\varepsilon$, if no confusion
arises.

Its dimension and orientability can be obtained in the same way as the ordinary Seiberg-Witten moduli space. The linearization $dH$ is deformed by a homotopy to
$$d^++2d^* : \Gamma(\Lambda^1)^G\rightarrow
\Gamma(\Lambda^0\oplus\Lambda^2_+)^G$$ and $$D_A :
\Gamma(W_+)^G\rightarrow \Gamma(W_-)^G$$ so that the virtual dimension of  $\mathfrak{X}$ is equal to $$\dim H_1(M;\Bbb R)^G-b_2^+(M)^G-1+2(\dim_{\Bbb C}(\ker D_A)^G-\dim_{\Bbb C}(\textrm{coker} D_A)^G),$$
and its orientation can be assigned by fixing the homology orientation of $H_1(M;\Bbb R)^G$ and $H_2^+(M;\Bbb R)^G$. When $G$ is finite, one can use Lefschetz-type formula to explicitly compute the last term $\textrm{ind}^G D_A$ in the above formula. For more details, one may consult \cite{cho}.

\begin{thm}\label{cpt}
When $G$ is finite, $\mathfrak{X}_\varepsilon$ for any $\varepsilon$ is compact.
\end{thm}
\begin{proof}
Following the proof for the ordinary Seiberg-Witten moduli space, we need the $G$-equivariant version of the gauge fixing lemma.
\begin{lem}
Let $\frak{L}$ be a $G$-equivariant complex line bundle over $M$  with a hermitian metric, and $A_0$ be a fixed $G$-invariant smooth unitary connection on it.

Then for any $l\geq 0$ there are constants $K, C>0$ depending on $A_0$ and $l$ such that for any $G$-invariant $L^2_l$ unitary connection $A$ on $\frak{L}$ there is a $G$-invariant $L^2_{l+1}$ change of gauge $\sigma$ so that $\sigma^*(A)=A_0+\alpha$ where $\alpha\in L^2_l(T^*M\otimes i\Bbb R)^G$ satisfies
$$d^*\alpha=0,\ \ \ \textrm{and}\ \ \ \ ||\alpha||^2_{L^2_l}\leq C||F^+_A||^2_{L^2_{l-1}}+K.$$
\end{lem}
\begin{proof}
We know that a gauge-fixing $\sigma$ with the above estimate always exists, but we need to prove the existence of $G$-invariant $\sigma$.
Write $A$ as $A_0+a$ where $a\in L^2_l(T^*M\otimes i\Bbb R)^G$. Let $a=a^{harm}+df+d^*\beta$ be the Hodge decomposition. By the $G$-invariance of $a$, so are  $a^{harm}, df$, and $d^*\beta$.
Applying the ordinary gauge fixing lemma to $A_0+d^*\beta$, we have $$||d^*\beta||^2_{L^2_l}\leq C'||F^+_{A_0+d^*\beta}||^2_{L^2_{l-1}}+K'=C'||F^+_A||^2_{L^2_{l-1}}+K'$$ for some constants $C',K'>0$.
Defining a $G$-invariant $i\Bbb R$-valued function $f_{av}=\frac{1}{|G|}\sum_{g\in G}g^*f$, we have
$$df=\frac{1}{|G|}\sum_{g\in G}g^*df=d(f_{av})=-d\ln \exp(-{f_{av})},$$ and hence $df$ can be gauged away by a $G$-invariant gauge transformation $\exp(-f_{av})$.
Write $a^{harm}$ as $(n|G|+m)a^{h}$ for $m\in [0,|G|)$ and an integer $n\geq 0$, where $a^h\in H^1(M;\Bbb Z)^G$ is not a positive multiple of another element of $H^1(M;\Bbb Z)^G$. There exists $\frak{g}\in \mathcal{G}$ such that $a^h=-d\ln \frak{g}$. In general $\frak{g}$ is not $G$-invariant, but $$|G|a^h=\sum_{g\in G}g^*a^h=-d\ln \prod_{g\in G}g^*\frak{g},$$ and hence $n|G|a^h$ can be gauged away by a $G$-invariant gauge transformation $\prod_{g\in G}g^*\frak{g}^n$.
In summary, $A_0+a$ is equivalent to $A_0+ma^{h}+d^*\beta$ after a $G$-invariant gauge transformation, and
\begin{eqnarray*}
||ma^{h}+d^*\beta||^2_{L^2_l}&\leq& (||ma^{h}||_{L^2_l}+||d^*\beta||_{L^2_l})^2\\ &\leq& |G|^2||a^{h}||_{L^2_l}^2+2|G|||a^{h}||_{L^2_l}||d^*\beta||_{L^2_l}+||d^*\beta||_{L^2_l}^2\\ &\leq& 3|G|^2||a^{h}||^2_{L^2_l}+ 3||d^*\beta||_{L^2_l}^2  \\   &=& K''+3C'||F^+_A||^2_{L^2_{l-1}}+3K'
\end{eqnarray*}
for a constant $K''>0$. This completes the proof.
\end{proof}
Now the rest of the compactness proof proceeds in the same way as the ordinary case using the Weitzenb\"ock formula and standard elliptic and Sobolev estimates. For details the readers are referred to \cite{morgan}.
\end{proof}

\begin{rmk}
If $G$ is not finite, $\mathfrak{X}_\varepsilon$ may not be compact.

For example, consider $M=S^1\times Y$ with the trivial Spin$^c$
structure and its obvious $S^1$ action, where $Y$ is a closed
oriented 3-manifold. For any $n\in\Bbb Z$, $n d\theta$ where
$\theta$ is the coordinate on $S^1$ is an $S^1$-invariant
reducible solution. Although $n d\theta$ is gauge equivalent to 0,
but never via an $S^1$-invariant gauge transformation which is an
element of the pull-back of $C^\infty(Y,S^1)$. Therefore as
$n\rightarrow \infty$, $n d\theta$ diverges modulo
$\mathcal{G}^{S^1}$, which proves that $\mathfrak{X}_0$ is
non-compact.
\end{rmk}

In the rest of this paper, we assume that $G$ is finite. Then note that $G$ induces smooth actions on $$\mathcal{C}:=\mathcal{A}(W_+)\times \Gamma(W_+),$$ $$\mathcal{B}^*=(\mathcal{A}(W_+)\times (\Gamma(W_+)-\{0\}))/ \mathcal{G},$$ and also the Seiberg-Witten moduli space $\frak{M}$ whenever it is smooth.

Since $\frak{X}_\varepsilon$ is a subset of $\frak{M}_\varepsilon$, (actually a subset of the fixed locus $\mathfrak{M}^G_\varepsilon$ of a $G$-space $\frak{M}_\varepsilon$), we can define the
$G$-monopole invariant $SW^G_{M,\mathfrak{s}}$ by integrating the same universal cohomology classes as in the ordinary Seiberg-Witten invariant $SW_{M,\mathfrak{s}}$. Thus using the $\Bbb Z$-algebra isomorphism $$\mu_{M,\frak{s}} : \Bbb Z[H_0(M;\Bbb Z)]\otimes \wedge^*H_1(M;\Bbb Z)/\textrm{torsion}\ \tilde{\rightarrow}\ H^*(\mathcal{B}^*;\Bbb Z),$$ 
we define it as a function $$SW^G_{M,\frak{s}}: \Bbb Z[H_0(M;\Bbb
Z)]\otimes \wedge^*H_1(M;\Bbb Z)/\textrm{torsion}\rightarrow \Bbb
Z$$ $$\alpha\mapsto \langle [\frak{X}],\mu_{M,\frak{s}}(\alpha)
\rangle,$$ which is set to be 0 when the degree of
$\mu_{M,\frak{s}}(\alpha)$ does not match $\dim \frak{X}$. To be
specific, for  $[c]\in H_1(M,\Bbb Z)$,
$$\mu_{M,\frak{s}}([c]):=Hol_c^*([d\theta])$$ where $[d\theta]\equiv 1\in H^1(S^1,\Bbb Z)$ and  $Hol_c: \mathcal{B}^*\rightarrow S^1$ is given by the holonomy of each connection around $c$, and  $\mu_{M,\frak{s}}(U)$  for $U\equiv 1\in H_0(M,\Bbb Z)$ is given by the first
Chern class of the $S^1$-bundle
$$\mathcal{B}^*_o=(\mathcal{A}(W_+)\times (\Gamma(W_+)-\{0\}))/
\mathcal{G}_o$$ over $\mathcal{B}^*$ where $\mathcal{G}_o=\{\frak{g}\in
\mathcal{G}| \frak{g}(o)=1\}$ is the based gauge group for a fixed base
point $o\in M$. (The $S^1$-bundles obtained by choosing a
different base point are all isomorphic by the connectedness of
$M$.)

As in the ordinary case, a different choice of a $G$-invariant
metric and a $G$-invariant perturbation $\varepsilon$ gives a
cobordant $\frak{X}$ so that $SW^G_{M,\mathfrak{s}}$ is
independent of such choices, if $b_2^{+}(M)^G> 1$. When
$b_2^{+}(M)^G= 1$, one should get an appropriate wall-crossing
formula.


When $\frak{M}$ happens to be smooth for a $G$-invariant
perturbation, the induced $G$-action on it is a smooth action, and
hence $\mathfrak{M}^G$ is a smooth submanifold. Moreover if the
finite group action is free, then $\pi: M\rightarrow M/G$ is a
covering, and $\frak{s}$ is the pull-back of a Spin$^c$ structure
on $M/G$, which is determined up to
the kernel of $\pi^*: H^2(M/G,\Bbb Z)\rightarrow H^2(M,\Bbb Z),$ and  all the irreducible solutions of the upstairs is precisely the pull-back of the corresponding irreducible solutions of the downstairs :
\begin{thm}[\cite{RW, naka}]\label{nakamur}
Let $M$, $\mathfrak{s}$, and $G$ be as above. Under the assumption that $G$ is finite and
the action is free, for a $G$-invariant generic perturbation $$\frak{X}_{M,\mathfrak{s}}=\mathfrak{M}_{M/G,\mathfrak{s}'} \ \ \ \ \textrm{and} \ \ \ \ \mathfrak{M}^G_{M,\mathfrak{s}}\backsimeq\coprod_{c\in \ker \pi^*}\mathfrak{M}_{M/G,\mathfrak{s}'+c},$$ where the second one is a homeomorphism in general, and  $\mathfrak{s}'$ is the Spin$^c$ structure on $M/G$ induced from  $\frak{s}$ and its $G$-action.
\end{thm}

Finally we remark that the $G$-monopole invariant may change when
a homotopically different lift of given $G$-action on $M$ to its Spin$^c$
structure is chosen.

\section{equivariant Bauer-Furuta invariant}
For a more refined invariant to find a $G$-monopole class, one can also consider equivariant Bauer-Furuta invariant in the same way as the ordinary Bauer-Furuta invariant. 

Let $A_0\in \mathcal{A}(W_+)^G$. Just as $$Pic(M):=(A_0+i\ker d)/ \mathcal{G}_o$$ is a $b_1(M)$-dimensional torus, one gets the quotient
$$Pic^G(M):=(A_0+i\ker d)^G/ \mathcal{G}_o^G.$$
\begin{lem}\label{truelove}
$Pic^G(M)$ is diffeomorphic to a torus $T^{b_1(M)^G}$ of dimension
$b_1(M)^G:=\dim H^1(M;\Bbb R)^G$, and also covers a torus
$T^{b_1(M)^G}$ embedded in $Pic(M)$.
\end{lem}
\begin{proof}
Here we need the condition that $G$ is finite.
Let $A_0+\alpha\in (A_0+i\ker d)^G$.

If $\alpha\in \textrm{Im}\ d$, namely $\alpha=df$ for some $f\in Map(M,i\Bbb R)$, then $$\alpha=df=d(\frac{\sum_{h\in G}h^*f}{|G|}),$$ and hence $\alpha\in  d\ln \mathcal{G}_o^G.$

If $[\alpha]$ defines a nonzero element in $H^1(M;i\Bbb Z)$, then write  $\alpha=d\ln \mathfrak{g}$ for $\mathfrak{g}\in\mathcal{G}_o$, and $$|G|\alpha=\sum_{h\in G}h^*d\ln \mathfrak{g} =d\ln \prod_{h\in G}h^*\mathfrak{g}\in d\ln \mathcal{G}_o^G.$$

Let $\{[\alpha_1],\cdots,[\alpha_{b_1(M)}]\}$ be a generating set for $H^1(M;i\Bbb Z)\simeq \Bbb Z^{b_1(M)}$.
For $[\alpha_i]\in H^1(M;i\Bbb Z)^G$, let $n_i$ be the smallest
positive number such that $n_i\alpha_i\in  d\ln \mathcal{G}_o^G$.
In fact, $n_i$ must be an integer. Thus if $b_1(M)^G\ne 0$, then
$Pic^G(M)$ is the obvious $m$-fold covering of the subtorus
generated by  those $[\alpha_i]$'s in $H^1(M;i\Bbb Z)^G$, where $m$
is the product of those $n_i$'s. If $b_1(M)^G= 0$, then obviously
$Pic^G(M)$ is a point embedded in $Pic(M)$.
\end{proof}

Define infinite-dimensional Hilbert bundles $\mathcal{E}^G$ and $\mathcal{F}^G$ over $Pic^G(M)$ by
$$\mathcal{E}^G:=\tilde{\mathcal{E}}^G/\mathcal{G}_o^G,\ \ \ \ \textrm{and}\ \ \ \ \ \mathcal{F}^G:=\tilde{\mathcal{F}}^G/\mathcal{G}_o^G,$$ where
$$\tilde{\mathcal{E}}^G:= (A_0+i\ker d)^G\times(\Gamma(W_+)^G\oplus \Gamma(\Lambda^1M)^G\oplus H^0(M)^G)$$ and
$$\tilde{\mathcal{F}}^G:= (A_0+i\ker d)^G\times(\Gamma(W_-)^G\oplus \Gamma(\Lambda^2_+M)^G\oplus L_{m}^2(\Lambda^0M)^G\oplus H^1(M)^G)$$ are endowed with appropriate Sobolov norms and a nontrivial $\mathcal{G}_o^G$ action  on the connection part $(A_0+i\ker d)^G$ and the spinor parts.

The $G$-monopole map $\mu^G : \mathcal{E}^G \rightarrow \mathcal{F}^G$ is an $S^1$-equivariant continuous fiber-preserving map defined as
$$[A,\Phi,a,f]\mapsto [A,D_{A+ia}\Phi, F_{A+ia}^+-\Phi\otimes\Phi^*+\frac{|\Phi|^2}{2}\textrm{Id}, d^*a+f, a^{harm}],$$ which is fiberwisely the sum of a linear Fredholm operator denoted by $\mathfrak{L}^G$ and a (quadratic) compact operator. Note that $$(\mu^G)^{-1}(\textrm{zero section of }\mathcal{F}^G)/S^1$$ is exactly the $G$-monopole moduli space.  The important property that the inverse image of any bounded set in $\mathcal{F}^G$  is bounded follows directly from the corresponding boundedness property of the ordinary monopole map $\mu: \mathcal{E} \rightarrow \mathcal{F}$ with linear part $\mathfrak{L}$. (This notation $\mu$ should not be confused with the $\mu$ map when defining Seiberg-Witten invariants in previous sections.)

To express the $G$-monopole map as an $S^1$-equivariant stable cohomotopy class, we take  an $S^1$-equivariant trivialization $\mathcal{F}^G\simeq Pic^G(M)\times \mathcal{U}^G$ with the projection map $\pi : \mathcal{F}^G\rightarrow \mathcal{U}^G$, and  take finite-dimensional approximations of  $$\pi\circ \mu^G : \mathcal{E}^G\rightarrow \mathcal{U}^G.$$ The virtual index bundle $\textrm{ind}\ \mathfrak{L}^G$ over $Pic^G(M)$ is $$\ker (D)^G-\textrm{coker}(D)^G-\underline{H^2_+(M)^G}\in KO(Pic^G(M)),$$ where $D$ is the Spin$^c$ Dirac operator, and $\underline{H^2_+(M)^G}$ is the trivial bundle of rank $b_2^+(M)^G:=\dim H^2_+(M;\Bbb R)^G$.
Note that  $\textrm{ind}\ \mathfrak{L}^G$ can be represented as $$E-F\in KO(Pic^G(M)),$$ where $E:=(\mathfrak{L}^G)^{-1}(F)$, and $F:=Pic^G(M)\times V$ for a finite dimensional subspace $V\subset \mathcal{U}^G$.

With $TH$ denoting the Thom space of a vector bundle $H$, define an $S^1$-equivariant stable cohomotopy group
\begin{eqnarray}\label{bf2}
\pi^0_{S^1,\mathcal{U}}(Pic^G(M);\textrm{ind}\ \mathfrak{L}^G)
\end{eqnarray}
as
$$\textrm{colim}_{U\subset \mathcal{U}^G} [S^U\wedge TE,S^U\wedge TF]^{S^1},$$
where $U$ runs all finite dimensional real vector subspaces of $\mathcal{U}^G$ transversal to $V$, 
and  $S^U\wedge $  denotes the smash product with the one-point compactification of a vector space $U$.

Then our $G$-monopole map gives an element $\overline{BF}^G_{M,\frak{s}}$ in the above stable cohomotopy group.
When $G$ is the trivial group $\{1\}$, $\overline{BF}^{\{1\}}_{M,\frak{s}}$ is just equal to the ordinary Bauer-Furuta invariant $BF_{M,\frak{s}}$ in \cite{BF, bau}. Just as $BF_{M,\frak{s}}$, $\overline{BF}^G_{M,\frak{s}}$ can be also viewed as the $S^1$-equivariant homotopy class of  $\mu^G$ in the set of the $S^1$-equivariant continuous fiber-preserving maps which differ from $\mu^G$ by the fiberwise compact perturbations and have bounded inverse image for any bounded subset in $\mathcal{F}^G$. (See \cite{bauer}.)
An important fact for our purpose is the following :
\begin{thm}
If $\overline{BF}^G_{M,\frak{s}}$ is nontrivial, then $c_1(\frak{s})$ is a $G$-monopole class.
\end{thm}
\begin{proof}
This is a consequence of facts from functional analysis, and one can take the proof in \cite[Proposition 6]{IL} verbatim, which proves that $c_1(\frak{s})$ is a monopole class, if ordinary Bauer-Furuta invariant $BF_{M,\frak{s}}\ne 0$.
\end{proof}

The $G$-equivariant Bauer-Furuta invariant $BF^G_{M,\frak{s}}$
first introduced (in case of $b_1(M)=0$) by M. Szymik
\cite{szymik} is a little different. (See also \cite{naka1}.) The
ordinary monopole map $\mu : \mathcal{E} \rightarrow \mathcal{F}$
is $(G\times S^1)$-equivariant, and one takes its class in the
$(G\times S^1)$-equivariant stable homotopy group
\begin{eqnarray}\label{bf1}
\pi^0_{G\times S^1,\mathcal{U}}(Pic(M);\textrm{ind}\ \mathfrak{L})
\end{eqnarray}
to get $BF^G_{M,\frak{s}}$. There is the obvious forgetful map from (\ref{bf1}) to
\begin{eqnarray*}
\pi^0_{S^1,\mathcal{U}}(Pic(M);\textrm{ind}\ \mathfrak{L}),
\end{eqnarray*}
under which $BF^G_{M,\frak{s}}$ gets mapped to $BF_{M,\frak{s}}$.

\begin{lem}\label{saveyou}
If the fixed point set $M^G$ is nonempty  or $b_1(M)^G=0$, then
the obvious quotient maps $$p_1 : \mathcal{E}^G\rightarrow
\tilde{\mathcal{E}}^G/\mathcal{G}_o\ \ \ \ \textrm{and}\ \ \ \ p_2
: \mathcal{F}^G\rightarrow \tilde{\mathcal{F}}^G/\mathcal{G}_o$$
are bijective, and $Pic^G(M)$ is a submanifold of $Pic(M)$.
\end{lem}
\begin{proof}
Since $\mathcal{G}_o^G$ is a subgroup of $\mathcal{G}_o$, $p_1$
and $p_2$ are obviously surjective.

To show that $p_1$ is injective, suppose that
$[A_1,\Phi_1,a_1,f_1]$ and $[A_2,\Phi_2,a_2,f_2]$ in
$\mathcal{E}^G$ are equivalent under $\gamma\in\mathcal{G}_o$.
Then $$A_1=A_2-2d\ln \gamma, \ \ \ \textrm{and}\ \ \
\Phi_1=\gamma\Phi_2.$$ By the first equality, $d\ln \gamma$ is
$G$-invariant.

Let's first consider the case when $M^G\ne \emptyset$.
Let $S$ be the subset of $M$ where $\gamma$ is $G$-invariant. By the continuity of $\gamma$, $S$ must be a closed subset. Since $S$ contains $M^G\ne \emptyset$, $S$ is nonempty.
It suffices to show that ${S}$ is open. Let $x_0\in {S}$. Then we have that for any $g\in G$, $$g^*\ln\gamma(x_0)=\ln\gamma(x_0), \ \ \ \textrm{and}\ \ \ g^*d\ln\gamma=d\ln\gamma,$$ which implies that $g^*\ln\gamma=\ln\gamma$ on an open neighborhood of $x_0$ on which $g^*\ln\gamma$ and $\ln\gamma$ are well-defined. By the compactness of $G$, there exists an open neighborhood of $x_0$ on which $g^*\ln\gamma$ is well-defined for all $g\in G$, and $\ln \gamma$ is $G$-invariant. This proves the openness of $S$.

In case when  $b_1(M)^G=0$, a $G$-invariant closed 1-form $d\ln
\gamma$ can be written as $df$ for $f\in Map(M,i\Bbb R)$. Again
using the compactness of $G$, $df=d(\frac{\sum_{h\in
G}h^*f}{|G|})$, and so $\gamma\in\mathcal{G}_o^G$.

In the same way, one can show that $p_2$ is injective.

Now it follows that $Pic^G(M)$ becomes a submanifold of $Pic(M)$.
Namely, the $m$ in Lemma \ref{truelove} is 1.
\end{proof}

Thus if $M^G\ne \emptyset$ or $b_1(M)^G=0$, then $\mathcal{E}^G$
and $\mathcal{F}^G$ are subsets of $\mathcal{E}$ and $\mathcal{F}$
respectively so that we can think of the restriction of $\mu$ to
$\mathcal{E}^G$, which is equal to $\mu^G$. Letting $\rho$ be the
map from (\ref{bf1}) to (\ref{bf2}) induced by restricting to its
$G$-fixed point set,  we have :
\begin{thm}
If $M^G\ne \emptyset$  or $b_1(M)^G=0$, then
$$\rho(BF^G_{M,\frak{s}})=\overline{BF}^G_{M,\frak{s}}.$$
\end{thm}
As observed in \cite{szymik}, $\rho$ is not injective in general.
But whether $Pic^G(M)$ is a submanifold of $Pic(M)$ or not, we can
conclude the following important fact :
\begin{thm}
If  $\overline{BF}^G_{M,\frak{s}}$ is not zero, then so is
$BF^G_{M,\frak{s}}$.
\end{thm}
\begin{proof}
Assume  $BF^G_{M,\frak{s}}$ is zero. Let $pr$ be the covering map
from $Pic^G(M)$ onto a torus $T^{b_1(M)^G}\subset Pic(M)$ as shown
in Lemma \ref{truelove}, and
$$\sigma : \pi^0_{G\times S^1}(Pic(M);\textrm{ind}\
\mathfrak{L})\rightarrow \pi^0_{G\times
S^1}(pr(Pic^G(M));\textrm{ind}\ \mathfrak{L})$$  be the
restriction map to the fibers over $pr(Pic^G(M))$. Then
$\sigma(BF^G_{M,\frak{s}})$ is also zero.

Further restricting $\sigma(BF^G_{M,\frak{s}})$ to the $G$-fixed
point set in each fiber over $Pic^G(M)$, we get an element
$\rho(\sigma(BF^G_{M,\frak{s}}))$ in
$\pi^0_{S^1}(pr(Pic^G(M));\textrm{ind}\ \mathfrak{L}^G)$. (By
abuse of notation, we still denote this restriction map by
$\rho$.) Thus we have $\rho(\sigma(BF^G_{M,\frak{s}}))=0$.

The covering map $pr$ induces a lifting map from
$\pi^0_{S^1}(pr(Pic^G(M));\textrm{ind}\ \mathfrak{L}^G)$ to
$\pi^0_{S^1}(Pic^G(M);\textrm{ind}\ \mathfrak{L}^G)$ in an obvious
way, and it maps $\rho(\sigma(BF^G_{M,\frak{s}}))$ to
$\overline{BF}^G_{M,\frak{s}}$. Thus finally we get the vanishing
of $\overline{BF}^G_{M,\frak{s}}$, yielding a contradiction.
\end{proof}

When the $G$-action is free, $\overline{BF}^G_{M,\frak{s}}$ is equal to $BF_{M/G,\frak{s}'}$, where $\frak{s}'$ is the Spin$^c$ structure on $M/G$ induced from $\frak{s}$ and its $G$-action. Under the further assumption that $|G|$ is prime, and the dimension of Seiberg-Witten moduli space is zero, the equivariant Bauer-Furuta invariant may be expressed as $BF_{M,\frak{s}}$ and $BF_{M/G,\frak{s}''}$ for all $\frak{s}''$ lifting to $\frak{s}$. (See \cite{szymik}.) In general, it is difficult to compute $BF^G_{M,\frak{s}}$ as well as $BF_{M,\frak{s}}$ itself. Therefore it is quite worthwhile  to compute $\overline{BF}^G_{M,\frak{s}}$ when the $G$-action is not free.


\section{Connected sums and $\Bbb Z_k$-monopole invariant}

For $(\bar{M}_k,\bar{\frak{s}})$ described in Theorem
\ref{firstth}, there is at least one obvious $\Bbb Z_k$-action on
$\bar{\frak{s}}$ coming from the given $\Bbb Z_k$-action on
$\frak{s}_N$ and the $\Bbb Z_k$-equivariant gluing of $k$-copies
of $\frak{s}$. We will call such an action ``canonical" and denote
its generator by $\tau$. Any other lifting of the $\Bbb
Z_k$-action on $\bar{M}_k$ is given by a group generated by
$\frak{g}\circ \tau$ where $\frak{g}$ is a gauge transformation of
$\bar{\frak{s}}$, i.e. the multiplication by a smooth $S^1$-valued
function on $\bar{M}_k$.

In general, there may be homotopically inequivalent liftings of the $\Bbb Z_k$-action on $\bar{M}_k$, but we have
\begin{lem}\label{triv}
Any lifted action on $\bar{\frak{s}}$ can be homotoped to another lifting which is equal to a canonical lifting on the cylindrical gluing regions.
\end{lem}
\begin{proof}
Let $\sigma$ be a generator of a $\Bbb Z_k$ action on $\bar{\frak{s}}$ covering the $\Bbb Z_k$-action on $\bar{M}_k$.
For the trivialization of the Spin$^c$ structure on the $k$ cylindrical regions $\cup_{j=1}^k U_j$ where each $U_j$ is diffeomorphic to $S^3\times [0,1]$ such that the generator $\tau$ of a canonical action acts as the multiplication by 1 there, $$\sigma|_{U_j} : \bar{\frak{s}}|_{U_j}\rightarrow \bar{\frak{s}}|_{U_{j+1}}$$ covering the identity map from $U_j$ to $U_{j+1}$ is given by the multiplication by $$e^{i\sigma_j}\in C^\infty(S^3\times [0,1],S^1).$$ Since $\sigma^k=Id$, $$\sum_{j=1}^k\sigma_j\equiv 0\ \ \ \textrm{mod}\ 2\pi.$$

By the fact that $H^1(S^3\times [0,1];\Bbb Z)=0$, any two gauge transformations on $S^3\times [0,1]$ can be homotoped to each other, and hence one can easily see that there exists a smooth relative  homotopy $$H_j(x,t): (S^3\times [0,1])\times [0,1]\rightarrow S^1$$ such that $$H_j(\cdot,0)=e^{i\sigma_j(\cdot)},$$  $$H_j(x,\cdot)=e^{i\sigma_j(x)}\ \ \textrm{for all } x\in S^3\times \{0,1\},$$ $$H_j(\cdot,1)|_{S^3\times [\frac{1}{3},\frac{2}{3}]\times \{1\}}=1.$$ But $\prod_{j=1}^kH_j$ may not be 1 to fail to form a group $\Bbb Z_k$. To remedy this, modify only $H_k$ by $$\tilde{H}_k:=H_k\cdot\overline{\prod_{j=1}^kH_j}$$ which is also a smooth relative homotopy satisfying the above three properties and produces $$(\prod_{j=1}^{k-1}H_j)\tilde{H}_k=1.$$
Therefore we have obtained a homotopy of the initial action to the action which is trivial on each $S^3\times [\frac{1}{3},\frac{2}{3}]$, while keeping equal to the initial action outside $\cup_{j=1}^k U_j$.
\end{proof}

By our assumption of $b_2^+(\bar{M}_k)^{\Bbb Z_k}=b_2^+(M)>1$, this homotopy of the action does not change the $\Bbb Z_k$-monopole invariant.
From now on we always assume that the $\Bbb Z_k$-action
on $\bar{\frak{s}}$ is equal to a canonical action on the
cylindrical gluing regions, and take a trivialization of the
Spin$^c$ structure of $k$ cylindrical gluing regions so that a
canonical action $\tau$ acts as the identity there.

The following theorem gives a proof for the $G$-monopole invariant of $(\bar{M}_k,\bar{\frak{s}})$ in Theorem \ref{firstth}.
Before stating the theorem, note that $$\textrm{rank}(H_1(N;\Bbb Z)^{\Bbb Z_k})=\dim H_1(N;\Bbb R)^{\Bbb Z_k},$$ simply because $\Bbb Z_k$ also acts on $H_1(N;\Bbb Z)$.
\begin{thm}\label{myLord}
Let $(\bar{M}_k,\bar{\frak{s}})$ be as in Theorem \ref{firstth} and $d\geq 0$ be an integer.

If $\nu:=\dim H_1(N;\Bbb R)^{\Bbb Z_k}= 0$, then for $A=1$ or
$a_1\wedge\cdots\wedge a_{j}$
$$SW^{\Bbb Z_k}_{\bar{M}_k,\bar{\frak{s}}}(U^d A)\equiv
SW_{M,\frak{s}}(U^d A)\ \ \ mod \ 2,$$ where $U$ denotes the
positive generator of the zeroth homology of $\bar{M}_k$ or $M$,
and each $a_i\in H_1(M;\Bbb Z)/\textrm{torsion}$ also denotes any
of $k$ corresponding elements in $H_1(\bar{M}_k;\Bbb Z)$ by abuse
of notation.

If $\nu\ne 0$, then
$$SW^{\Bbb Z_k}_{\bar{M}_k,\bar{\frak{s}}}(U^d A\wedge b_1\wedge\cdots\wedge b_\nu)\equiv
SW_{M,\frak{s}}(U^d A)\ \ \ mod \ 2,$$ where $A$ is as above, and
$b_1,\cdots, b_\nu\in H_1(N;\Bbb Z)$ is a basis of $H_1(N;\Bbb R)^{\Bbb
Z_k}$.
\end{thm}
\begin{proof}
First we consider the case when the action on $N$ has a fixed point.

We take a $\Bbb Z_k$-invariant metric of positive scalar curvature on $N$. In order to do the connected sum with $k$ copies of $M$, we perform a Gromov-Lawson type surgery \cite{GL,sung1} around each  point of a free orbit of $\Bbb Z_k$ keeping the positivity of scalar curvature to get a Riemannian manifold $\hat{N}$ with cylindrical ends with each end isometric to a Riemannian product of a round $S^3$ and $\Bbb R$. We suppose that this is done in a symmetric way so that the $\Bbb Z_k$-action on $\hat{N}$ is isometric.

On $M$ part, we put any metric and perform a Gromov-Lawson surgery with the same cylindrical end as above. Let's denote this by $\hat{M}$. Now chop the cylinders at sufficiently large length and then glue $\hat{N}$ and $k$-copies of $\hat{M}$ along the boundary to get a desired  $\Bbb Z_k$-invariant metric $g_k$ on $\bar{M}_k$. Sometimes we mean $(\bar{M}_k,g_k)$ by $\bar{M}_k$.

Let's first figure out the ordinary moduli space $\frak{M}_{\bar{M}_k}$ of
$(\bar{M}_k,\bar{\mathfrak{s}})$. Let $\frak{M}_{\hat{M}}$ and
$\frak{M}_{\hat{N}}$ be the moduli spaces of finite-energy
solutions of Seiberg-Witten equations on $(\hat{M},\mathfrak{s})$ and
$(\hat{N},\mathfrak{s}_N)$ respectively. From now on, $[\ \cdot\ ]$ of a configuration $\cdot$ denotes its gauge equivalence class.

By the gluing theory\footnote{For more details, one may consult
\cite{KM, nicol, safari, vid1, sung2}.} of Seiberg-Witten moduli
space, which is now a standard method in gauge theory, $\frak{M}_{M}$
is diffeomorphic to $\frak{M}_{\hat{M}}$. In
$\frak{M}_{\hat{M}}$, we use a compact-supported self-dual 2-form
for a generic perturbation.

Since $\hat{N}$ has a metric of positive scalar curvature and the property that $b_2^+(\hat{N})=0$ and $c_1^2(\frak{s}_{\hat{N}})=-b_2(\hat{N})$,  $\hat{N}$ also has no gluing obstruction even without perturbation so that $\frak{M}_N$ is diffeomorphic to
$\frak{M}_{\hat{N}}=\frak{M}_{\hat{N}}^{red},$  which can be identified with the sspace of $L^2$-harmonic 1-forms on $\hat{N}$ modulo gauge, i.e. $$H^1_{cpt}(\hat{N},\Bbb R)/H^1_{cpt}(\hat{N},\Bbb Z)\simeq T^{b_1(N)}.$$ (Here by $T^0$ we mean a point, and $\frak{M}^{red}\subset \frak{M}$ denotes the moduli space of reducible solutions.)

As is well-known, approximate solutions on $\bar{M}_k$ are
obtained by chopping-off solutions on each $\hat{M}$ and $\hat{N}$
at a sufficiently large cylindrical length and then grafting them
to $\bar{M}_k$ via a sufficiently slowly-varying partition of
unity in a $\Bbb Z_k$-invariant way.

In taking cut-offs of solutions on $\hat{N}$, we use a special gauge-fixing condition. Fix a $\Bbb Z_k$-invariant connection $\eta_0$ such that $[\eta_0]\in\frak{M}_{\hat{N}}$,
which exists by taking the $\Bbb Z_k$-average of any reducible
solution, and take compact-supported closed 1-forms
$\beta_1,\cdots,\beta_{b_1(N)}$ which generate $H^1_{cpt}(\hat{N};\Bbb Z)$ and vanish on the cylindrical gluing regions. Any element $[\eta]\in\frak{M}_{\hat{N}}$ can be
expressed as $$\eta=\eta_0+\sum_{i}c_i\beta_i$$ for $c_i\in \Bbb
R/ \Bbb Z$, and the gauge equivalence class of its cut-off
$$\tilde{\eta}:=\rho\eta=\rho\eta_0+\sum_{i}c_i\beta_i$$ using a
$\Bbb Z_k$-invariant cut-off function $\rho$ which is equal to 1
on the support of every $\beta_i$ is well-defined independently of
the mod $\Bbb Z$ ambiguity of each $c_i$.

Similarly, for the cut-off procedure to be well-defined independently of the choice of a gauge representative on $\frak{M}_{\hat{M}}$, one needs to take a gauge-fixing so that homotopy classes of gauge transformations on $\hat{M}$ are parametrized  by $H^1_{cpt}(\hat{M},\Bbb Z)$, whose elements are gauge transformations constant on gluing regions.
Thus the gluing produces a smooth map from $$(\prod_{i=1}^k\frak{M}_{\hat{M}})\times \frak{M}_{\hat{N}}:=(\underbrace{\frak{M}_{\hat{M}}\times \cdots \times
\frak{M}_{\hat{M}}}_k)\times \frak{M}_{\hat{N}}$$
to a so-called approximate moduli space
$\tilde{\frak{M}}_{\bar{M}_k}$ in $\frak{B}^*_{\bar{M}_k}$. This gluing map is one to one, because of the unique continuation principle (\cite{KM}) of Seiberg-Witten equations. From the
unobstructedness of gluing, $\tilde{\frak{M}}_{\bar{M}_k}\subset
\frak{B}^*_{\bar{M}_k}$ is a smoothly embedded submanifold diffeomorphic to
\begin{eqnarray*}
((\prod_{i=1}^k\frak{M}_{\hat{M}}^o)/S^1)\times\frak{M}_{\hat{N}} &=&
((\prod_{i=1}^k\frak{M}_{\hat{M}})\tilde{\times} T^{k-1})\times T^{b_1(N)},
\end{eqnarray*}
where $\frak{M}_{\hat{M}}^o$ is the based moduli space fibering
over $\frak{M}_{\hat{M}}$ with fiber $\mathcal G_o/ \mathcal
G=S^1$, and $\tilde{\times}$ means a $T^{k-1}$-bundle over
$\prod_{i=1}^k\frak{M}_{\hat{M}}$.

As the length of the cylinders in $\bar{M}_k$ increases,
approximate solutions get close to genuine solutions
exponentially. Once we choose smoothly-varying normal subspaces to
tangent spaces of $\tilde{\frak{M}}_{\bar{M}_k}\subset
\frak{B}^*_{\bar{M}_k}$, the Newton method gives a diffeomorphism
$$\Upsilon : \tilde{\frak{M}}_{\bar{M}_k}\rightarrow
\frak{M}_{\bar{M}_k}$$ given by a very small isotopy along the
normal directions. A bit more explanation will be given in Lemma
\ref{saveme}.

An important fact for us is that the same $k$ copies of a compactly supported self-dual 2-form  can be used for the perturbation on $M$ parts, while no perturbation is put on the $N$ part. Along with the $\Bbb Z_k$-invariance of the Riemannian metric $g_k$, the perturbed Seiberg-Witten equations for $(\bar{M}_k,g_k)$ are $\Bbb Z_k$-equivariant so that the induced smooth $\Bbb Z_k$-action on $\mathcal{B}^*_{\bar{M}_k}$ maps $\frak{M}_{\bar{M}_k}$ to itself.

Let's describe elements of $\tilde{\frak{M}}_{\bar{M}_k}$ for
$(\bar{M}_k,g_k)$ more explicitly.  For $[\xi]\in
\frak{M}_{\hat{M}}$, let $\tilde{\xi}$ be an approximate solution
for $\xi$ cut-off at a large cylindrical length, and
$\tilde{\xi}(\theta)$ be its gauge-transform under the gauge
transformation by $e^{i\theta}\in C^\infty(\hat{M},S^1)$. (From
now on, the tilde $\tilde{\ }$ of a solution will mean its
cut-off.)  Any element in $\tilde{\frak{M}}_{\bar{M}_k}$ can be
written as an ordered $(k+1)$-tuple
$$[(\tilde{\xi}_1(\theta_1),\cdots , \tilde{\xi}_k(\theta_k),
\tilde{\eta})]$$ for each $[\xi_i]\in \frak{M}_{\hat{M}}$ and
constants $\theta_i$'s, where the $i$-th term for $i=1,\cdots , k$
represents the approximate solution grafted on the
$i$-th $M$ summand, and the last term is a cut-off of $\eta\in
\frak{M}_{\hat{N}}^{red}$. Here, the
grafting over each $M$ part is done via the identification of the
Spin$^c$ structure of each $M$ part using a canonical action, and hence
a canonical action $\tau$ on it is given by the permutation of
approximate solutions on $M$ parts, i.e.
$$\tau^*(\tilde{\xi}_1(\theta_1),\cdots,\tilde{\xi}_k(\theta_k),
\tilde{\eta})=(\tilde{\xi}_k(\theta_k),\tilde{\xi}_1(\theta_1),\cdots
, \tilde{\xi}_{k-1}(\theta_{k-1}), \tau^*\tilde{\eta}).$$ The quotient of these
$(k+1)$-tuples by $S^1$ gives precisely $\tilde{\frak{M}}_{\bar{M}_k}$ so that there is a bijective
correspondence
 \begin{eqnarray}\label{general}
\tilde{\frak{M}}_{\bar{M}_k}
\end{eqnarray}
\begin{eqnarray*}
\wr|
\end{eqnarray*}
$$\{ [(\tilde{\xi_1}(\theta_1),\cdots , \tilde{\xi}_{k-1}(\theta_{k-1}),\tilde{\xi_k}(0), \tilde{\eta})]\ |\ [\eta]\in \frak{M}_{\hat{N}},   [\xi_i]\in \frak{M}_{\hat{M}}, \theta_i\in [0,2\pi)\ \forall i
\}.
$$
Let's denote a generator of the $\Bbb Z_k$-action $(\bar{M}_k,\bar{\frak{s}})$ by $\sigma$. By Lemma \ref{triv}, $\sigma$ is equal to 1 on the cylindrical gluing regions, and hence $\sigma$ can be obviously extended to an action on the Spin$^c$ structure of $\hat{N}\cup \amalg_{i=1}^k\hat{M}$ and also its moduli space of finite-energy monopoles.
By the $\Bbb Z_k$-invariance of $\rho$ $$\sigma^*\tilde{\eta}=\sigma^*(\rho\eta)=\rho\sigma^*\eta=\widetilde{\sigma^*\eta},$$ and thus
\begin{eqnarray}\label{mom}
\\
\sigma^*(\tilde{\xi}_1(\theta_1),\cdots , \tilde{\xi}_k(0),\tilde{\eta})
&=&(\tilde{\xi}_k(\sigma_k),\tilde{\xi}_1(\theta_1+\sigma_1),\cdots , \tilde{\xi}_{k-1}(\theta_{k-1}+\sigma_{k-1}),\widetilde{\sigma^*\eta}),\nonumber
\end{eqnarray}
where all $e^{i \sigma_i}\in C^\infty(\hat{M},S^1)$ are equal to 1 on the cylindrical gluing regions and satisfy $$\sum_{i=1}^k\sigma_i\equiv 0\ \ \textrm{mod}\ 2\pi.$$

\begin{lem}
The $\Bbb Z_k$-action on $\mathcal{B}^*_{\bar{M}_k}$ maps $\tilde{\frak{M}}_{\bar{M}_k}$ to itself.
\end{lem}
\begin{proof}
Since $\sigma^*\beta_i$ also gives an element of $H^1_{cpt}(\hat{N};\Bbb Z)$, let's let $\sigma^*\beta_i$ be cohomologous to $\sum_j d_{ij}\beta_j$ for each $i$. Thus
$$\widetilde{\sigma^*\eta}=\rho\sigma^*(\eta_0+\sum_{i}c_i\beta_i)=\rho\eta_0+\sum_{i}c_i\sigma^*\beta_i$$ is gauge-equivalent to $$\rho\eta_0+\sum_{i,j}c_id_{ij}\beta_j$$ which is the cut-off of $\eta_0+\sum_{i,j}c_id_{ij}\beta_j.$ Also using the fact that all $e^{i\sigma_j}$ are also 1 on the cylindrical gluing regions so that they can be gauged away by a global gauge transformation on $\bar{M}_k$ without affecting $\widetilde{\sigma^*\eta}$, the RHS of (\ref{mom}) is gauge-equivalent to
\begin{eqnarray}\label{papa}
(\tilde{\xi}_k(0),\tilde{\xi}_1(\theta_1),\cdots ,
\tilde{\xi}_{k-1}(\theta_{k-1}),\rho(\eta_0+\sum_{i,j}c_id_{ij}\beta_j))
\end{eqnarray}
which is an approximate solution.
\end{proof}

Moreover we may assume that the map $\Upsilon$ is $\Bbb
Z_k$-equivariant by the following lemma.
\begin{lem}\label{saveme}
$\Upsilon$ can be made $\Bbb Z_k$-equivariant, and the smooth submanifold $\frak{M}_{\bar{M}_k}^{\Bbb Z_k}$ pointwisely fixed under the action is isotopic to $\tilde{\frak{M}}_{\bar{M}_k}^{\Bbb Z_k}$, the fixed point set in $\tilde{\frak{M}}_{\bar{M}_k}$.
\end{lem}
\begin{proof}
To get a $\Bbb Z_k$-equivariant $\Upsilon$, we need to choose a
smooth normal bundle of
$\tilde{\frak{M}}_{\bar{M}_k}\subset\frak{B}^*_{\bar{M}_k}$ in a $\Bbb
Z_k$-equivariant way. This can be achieved by taking the $\Bbb
Z_k$-average of any smooth Riemannian metric defined in a small
neighborhood of $\tilde{\frak{M}}_{\bar{M}_k}$.

A smooth Riemannian metric on a Hilbert manifold is a smoothly
varying bounded positive-definite symmetric bilinear forms on its tangent spaces. In
order to have a well-defined exponential map as a diffoemorphism
on a neighborhood of the origin, we want the metric to be
``strong" in the sense that the metric on each tangent space
induces the same topology as the original Hilbert space topology.
(For a proof, see \cite{kling}.)

Since $\tilde{\frak{M}}_{\bar{M}_k}$ is compact, we use a
partition of unity on it to glue together obvious Hilbert space
metrics in local charts, thereby constructing a smooth Riemannian
metric in a neighborhood of $\tilde{\frak{M}}_{\bar{M}_k}$ in a
Hilbert manifold $\frak{B}^*_{\bar{M}_k}$. Taking its average under the $\Bbb
Z_k$-action, we get a desired Riemannian metric, which is easily
checked to be strong.

Taking the orthogonal complement to the tangent bundle of
$\tilde{\frak{M}}_{\bar{M}_k}$ under the above-obtained metric, we
get its normal bundle which is trivial by being
infinite-dimensional. In the same way as the finite dimensional case, the inverse function theorem implies that a small neighborhood of the zero section in the normal bundle is mapped
diffeomorphically into $\frak{B}^*_{\bar{M}_k}$ by the exponential map. Thus we can view a small
neighborhood of $\tilde{\frak{M}}_{\bar{M}_k}$ as
$\tilde{\frak{M}}_{\bar{M}_k}\times \Bbb H$  where  $\Bbb H$ is
the Hilbert space isomorphic to the orthogonal complement of the
tangent space of $\tilde{\frak{M}}_{\bar{M}_k}$ at any point.

Applying the Newton method, $\Upsilon$ is pointwisely a vertical
translation along $\Bbb H$
direction. Now the first assertion follows from the $\Bbb
Z_k$-invariance of the normal directions.
\end{proof}

As a preparation for finding $\Bbb Z_k$-fixed points of $\tilde{\frak{M}}_{\bar{M}_k}$,
\begin{lem}\label{adam}
$\frak{M}_{\hat{N}}^{\Bbb Z_k}$ is diffeomorphic to $T^{\nu}$, the space of $\Bbb Z_k$-invariant $L^2$-harmonic 1-forms on $\hat{N}$ modulo $\Bbb Z$.
\end{lem}
\begin{proof}
Let $[\eta]\in
\frak{M}_{\hat{N}}^{\Bbb Z_k}$, i.e. $[\sigma^*\eta]=[\eta].$ Then
$$\bar{\eta}:=\frac{1}{k}\sum_{i=1}^k(\sigma^i)^*\eta$$ satisfies that $\sigma^*\bar{\eta}=\bar{\eta}$, and $\bar{\eta}$ is cohomologous to $\eta$ so that  $[\bar{\eta}]=[\eta]$.

When $\nu\ne 0$, complete $b_1,\cdots,b_\nu$ to $b_1,\cdots,b_\nu,\cdots,b_{b_1(N)}\in H_1(N;\Bbb Z)$ so as to compose a basis for $H_1(N;\Bbb R)$, and let $b_1^*,\cdots,b_{b_1(N)}^*\in H^1_{cpt}(\hat{N};\Bbb R)$ be the corresponding dual cohomology classes under the isomorphism
$$H^1_{cpt}(\hat{N};\Bbb R)\simeq H_1(N;\Bbb R)^*.$$ Since
$b_i^*(b_j)=\delta_{ij}$ for all $i,j=1,,\cdots,b_1(N)$,  a simple Linear algebra shows that $b_1^*,\cdots,b_\nu^*$ are not only in
$H^1_{cpt}(\hat{N};\Bbb Z)^{\Bbb Z_k}$, but also form a basis of $H^1_{cpt}(\hat{N};\Bbb
R)^{\Bbb Z_k}$. Therefore $\frak{M}_{\hat{N}}^{\Bbb Z_k}$  is a
$\nu$-dimensional torus spanned by $b_1^*,\cdots,b_\nu^*$.
When $\nu=0$, $\frak{M}_{\hat{N}}^{\Bbb Z_k}$  is a point.
\end{proof}

\begin{lem}
$\tilde{\frak{M}}_{\bar{M}_k}^{\Bbb Z_k}$ is diffeomorphic to $k$ copies of $\frak{M}_{{M}}\times T^{\nu}$, where $T^0$ means a point.
\end{lem}
\begin{proof}
By (\ref{papa}), the condition for a fixed point is that
$$(\tilde{\xi}_k(0),\tilde{\xi}_1(\theta_1),\cdots,
\tilde{\xi}_{k-1}(\theta_{k-1}),\widetilde{\sigma^*\eta})\equiv
(\tilde{\xi}_1(\theta_1),\cdots,\tilde{\xi}_{k-1}(\theta_{k-1}),\tilde{\xi}_k(0),\tilde{\eta}
)$$ modulo gauge transformations. By (\ref{general}) this implies
$$[\xi_{1}]=[\xi_{2}]=\cdots =[\xi_{k}]   \in\frak{M}_{\hat{M}},\ \textrm{and }\ [\sigma^*{\eta}]=[{\eta}]\in \frak{M}_{\hat{N}},$$
and
$$0 \equiv \theta_1+\theta,\ \ \theta_1 \equiv \theta_2+\theta,\cdots, \theta_{k-1}\equiv 0+\theta\ \ \  \textrm{mod}\ 2\pi$$
for some constant $\theta\in [0,2\pi)$. Summing up the above $k$ equations gives $$0\equiv k\theta\ \ \  \textrm{mod}\ 2\pi,$$ and hence
$$\theta=0,\frac{2\pi}{k},\cdots,\frac{2(k-1)\pi}{k},$$ which lead to the corresponding  $k$ solutions
\begin{eqnarray}\label{temp}
[(\tilde{\xi}((k-1)\theta),\tilde{\xi}((k-2)\theta),\cdots
,\tilde{\xi}(\theta) ,
\tilde{\xi}(0),\tilde{\eta})],
\end{eqnarray}
where we let $\xi_i=\xi$ for all $i$ and $[\eta]\in \frak{M}_{\hat{N}}^{\Bbb Z_k}$.
Therefore $\tilde{\frak{M}}_{M_k}^{\Bbb Z_k}$ is diffeomorphic to $k$ copies of $\frak{M}_{\hat{M}}\times\frak{M}_{\hat{N}}^{\Bbb Z_k}\simeq\frak{M}_{M}\times T^{\nu}.$
\end{proof}

\begin{lem}
$\mathfrak{X}_{\bar{M}_k}$  is diffeomorphic to $\frak{M}_{M}\times T^{\nu}$.
\end{lem}
\begin{proof}
Let $$\tilde{\Xi}_{\theta}=(\tilde{\xi}((k-1)\theta-\sum_{i=1}^{k-1}\sigma_i),\tilde{\xi}((k-2)\theta-\sum_{i=2}^{k-1}\sigma_i),\cdots,\tilde{\xi}(\theta-\sigma_{k-1}), \tilde{\xi}(0),\tilde{\eta})$$ for $\xi\in \frak{M}_{\hat{M}}$, $\eta\in \mathfrak{X}_{\hat{N}}$, and $\theta$ as above. Because all $e^{i\sigma_j}$ are also 1 on the cylindrical gluing regions, the gauge-equivalence class $[\tilde{\Xi}_{\theta}]$ is equal to the above (\ref{temp}), and moreover it has a nice property $$\sigma^*\tilde{\Xi}_\theta=e^{i\theta}\cdot\tilde{\Xi}_\theta,$$ where $\cdot$ denotes the gauge action.
Let's denote $\Upsilon([\tilde{\Xi}_\theta])$ also by  $[\Xi_\theta]$.

We will show that $k-1$ copies of  $\frak{M}_{M}\times T^{\nu}$ corresponding to nonzero $\theta$ do not belong to  $\mathfrak{X}_{\bar{M}_k}$.
Let $\theta=\frac{2\pi}{k},\cdots,\frac{2(k-1)\pi}{k}$. By the $\Bbb Z_k$-equivariance of $\Upsilon$,
$[\sigma^*\Xi_\theta]=\sigma^*[\Xi_\theta]$, and so write  $$\sigma^*\Xi_\theta=e^{i\vartheta}\cdot\Xi_\theta \ \ \ \textrm{for}\ e^{i\vartheta}\in Map(\bar{M}_k,S^1).$$ By taking the cylindrical length sufficiently large,  $e^{i\vartheta}$ can be made arbitrarily close to the constant $e^{i\theta}$ in a Sobolov norm and hence $C^0$-norm too by the Sobolov embedding theorem. (The Sobelev embedding constant does not change, if the cylindrical length gets large, because the local geometries remain unchanged.)


Assume to the contrary that $\sigma^*(\frak{g}\cdot\Xi_\theta)=\frak{g}\cdot\Xi_\theta$ for some $\frak{g} \in \mathcal{G}$.
Then combined with that
\begin{eqnarray*}
\sigma^*(\frak{g}\cdot\Xi_\theta)&=& \sigma^*(\frak{g})\cdot\sigma^*(\Xi_\theta)\\ &=& \sigma^*(\frak{g})\cdot (e^{i\vartheta}\cdot\Xi_\theta)\\ &=& (\sigma^*(\frak{g})e^{i\vartheta})\cdot \Xi_\theta,
\end{eqnarray*}
it follows that
$$\frak{g}\cdot \Xi_\theta = (\sigma^*(\frak{g})e^{i\vartheta})\cdot \Xi_\theta,$$
which implies that
\begin{eqnarray}\label{prayer2}
\sigma^*(\frak{g})=\frak{g}e^{-i\vartheta}\approx \frak{g}e^{-i\theta},
\end{eqnarray}
where the equality is due to the continuity of $\frak{g}$ and the fact that the spinor part of $\Xi_\theta$ is not identically zero on an open subset by the unique continuation property, and the notation $\approx$ means $C^0$-close.

Choose a fixed point $p\in \bar{M}_k$under the $\Bbb
Z_k$-action.\footnote{This is the only place where we use the condition
that the action on $N$ has a fixed point, which was assumed in the
beginning of the proof of current theorem.}   Evaluating
(\ref{prayer2}) at the point $p$ gives $$\frak{g}(p)\approx \frak{g}(p)e^{-i\theta},$$ yielding a desired contradiction.

It remains to show that $\frak{M}_{M}\times T^{\nu}$ corresponding to $\theta=0$ belongs to $(\mathcal{A}(W_+)^G\times (\Gamma(W_+)^G-\{0\}))/ \mathcal{G}^G$.  Let $\Xi_0=\tilde{\Xi}_0+(a,\varphi)$ where $a\in \Gamma(\Lambda^1(\bar{M}_k;i\Bbb R))$ satisfies the Lorentz gauge condition $d^*a=0$. Since
$$\sigma^*\Xi_0=\tilde{\Xi}_0+(\sigma^*a,\sigma^*\varphi)$$ belongs to the same gauge equivalence class as  $\Xi_0$, and $$d^*(\sigma^*a)=\sigma^*(d^*a)=0$$ using the isometric action of $G$, we have that $\sigma^*a\equiv a$ modulo $H^1(\bar{M}_k;\Bbb Z)=\Bbb Z^{b_1(\bar{M}_k)}$. Applying the obvious identity $(\sigma^*)^k=\textrm{Id}$, it follows that $\sigma^*a=a$. This implies that $\sigma^*\Xi_0$ is a constant gauge transform $e^{ic}\cdot \Xi_0$ of $\Xi_0$, where $c$ is one of $0,\frac{2\pi}{k},\cdots,\frac{2(k-1)\pi}{k}$. Since $\sigma^*\Xi_0\approx \Xi_0$, $c$ has to be 0.
Therefore $\sigma^*\Xi_0=\Xi_0$ as desired, and we conclude that  $\mathfrak{X}_{\bar{M}_k}$ is equal to $\frak{M}_{M}\times T^{\nu}$ .
\end{proof}

\begin{lem}\label{LHW}
The $\mu$ cocycles on $\frak{M}_{{M}}\times T^\nu$ and  $\mathfrak{X}_{\bar{M}_k}$ coincide, i.e.
$$\mu_{M}(a_i)=\mu_{\bar{M}_k}(a_i),\ \ \ \  \mu_{N}(b_i)=\mu_{\bar{M}_k}(b_i),\ \ \ \ \mu_{M}(U)=\mu_{\bar{M}_k}(U)$$ where the equality means the identification under the above diffeomorphism.
\end{lem}
\begin{proof}
The first equality comes from that the holonomy
maps $Hol_{a_i}$ defined on $\frak{M}_{{M}}$ and
$\tilde{\frak{M}}_{\bar{M}_k}^{\Bbb Z_k}$ are just the same, when
the representative of $a_i$ is  chosen away from the gluing
regions.  Using the isotopy between  $\frak{M}_{\bar{M}_k}^{\Bbb
Z_k}$ and $\tilde{\frak{M}}_{\bar{M}_k}^{\Bbb Z_k}$, the induced
maps $Hol^*_{a_i}$ from $H^1(S^1;\Bbb Z)$ to $H^1(\frak{M}_{{M}};\Bbb Z)$ and
$H^1(\frak{M}_{\bar{M}_k}^{\Bbb Z_k};\Bbb Z)$ are the same so that
$$\mu_{M}(a_i)=Hol^*_{a_i}([d\theta])=\mu_{\bar{M}_k}(a_i)$$
for each $i$. Likewise for the second equality.

For the third equality, note that the $S^1$-fibrations on
$\frak{M}_{\hat{M}}\times T^\nu$ and
$\tilde{\frak{M}}_{\bar{M}_k}^{\Bbb Z_k}$ induced by the
$\mathcal{G}/\mathcal{G}_o$ action are isomorphic in an obvious
way, where the $T^\nu$ part is fixed under the
$\mathcal{G}/\mathcal{G}_o$ action. Since the isotopy between
$\tilde{\frak{M}}_{\bar{M}_k}$ and $\frak{M}_{\bar{M}_k}$ can be
extended to the $S^1$-fibrations  induced by  the
$\mathcal{G}/\mathcal{G}_o$ action, those $S^1$-fibrations are
isomorphic. In the same way using gluing theory, there are
isomorphisms of $S^1$-fibraions on  $\frak{M}_{M}$, its approximate moduli space
$\tilde{\frak{M}}_{M}$, and $\frak{M}_{\hat{M}}$. Therefore we
have an isomorphism between those $S^1$-fibrations on
$\frak{M}_{M}\times T^\nu$ and $\mathfrak{X}_{\bar{M}_k}$.
\end{proof}

Now we come to the evaluation of the Seiberg-Witten invariant
on $\mathfrak{X}_{\bar{M}_k}$.  Suppose $\nu\ne 0$. Let
$l_1,\cdots,l_{b_1(N)}$ be loops representing homology classes $b_1,\cdots,b_{b_1(N)}$
respectively.  Then $b_i^*$ introduced in Lemma \ref{adam}
restricts to a nonzero element  of
$H^1(l_j;\Bbb Z)$ iff $i=j$. Moreover $b_i^*$ is a generator of $H^1(l_j;\Bbb Z)$, and hence
$\{\mu(b_1),\cdots,\mu(b_\nu)\}$ is a standard generator of the 1st
cohomology of $T^\nu\simeq \Bbb R\langle
b_1^*,\cdots,b_\nu^*\rangle/\Bbb Z\langle
b_1^*,\cdots,b_\nu^*\rangle$. Combining the fact that $\mu(b_1)\wedge
\cdots \wedge \mu(b_{\nu})$ is a generator of $H^\nu(T^{\nu};\Bbb
Z)$ with the above identification of $\mu$-cocycles, we can conclude that
$$SW^{\Bbb Z_k}_{\bar{M}_k,\bar{\frak{s}}}(U^dA\wedge
b_1\wedge\cdots\wedge b_\nu)\equiv SW_{M,\frak{s}}(U^dA)\ \ \
\textrm{mod}\ 2$$ for $A=1$ or $a_1\wedge\cdots\wedge a_j$. The
case of $\nu=0$ is just a special case.

Finally let's prove the theorem when the action on $N$ is free. In this case we can apply Theorem \ref{nakamur} and the gluing theory, we get diffeomorphisms
\begin{eqnarray*}
\frak{X}_{\bar{M}_k,\bar{\mathfrak{s}}}&\simeq& \mathfrak{M}_{M\# N/\Bbb Z_k,\mathfrak{s}\# \frak{s}_N'}\\
 &\simeq& \mathfrak{M}_{M,\mathfrak{s}}\times \mathfrak{M}^{red}_{N/\Bbb Z_k, \frak{s}_N'}\\
 &\simeq& \mathfrak{M}_{M,\mathfrak{s}}\times T^\nu,
\end{eqnarray*}
where $\frak{s}_N'$ is the Spin$^c$ structure on $N/\Bbb Z_k$ induced from $\frak{s}_N$ and its $\Bbb Z_k$ action induced from that of $\bar{\frak{s}}$. The rest is the same as the non-free case.
This completes all the proof.
\end{proof}

\begin{rmk}
If the diffeomorphism between $\frak{X}_{\bar{M}_k}$ and $\mathfrak{M}_{M}\times T^\nu$
is orientation-preserving, then $\Bbb Z_k$-monopole invariants and Seiberg-Witten invariants are exactly the same.
We conjecture that the above diffeomorphism between $\frak{X}_{\bar{M}_k}$  and $\frak{M}_{M}\times T^{\nu}$ is orientation-preserving, when the homology orientations are appropriately chosen.

One may try to prove $\frak{X}_{\bar{M}_k}\simeq \frak{M}_{M}\times T^{\nu}$ by gluing $G$-monopole moduli spaces directly. But the above method of proof by gluing ordinary moduli spaces also shows that $\frak{M}_{\bar{M}_k}^{\Bbb Z_k}$ is diffeomorphic to $k$ copies of $\frak{M}_{M}\times T^{\nu}$.

Lemma \ref{LHW} is also true for any other component of $\frak{M}_{\bar{M}_k}^{\Bbb Z_k}$.
\end{rmk}

\section{Connected sums and $\Bbb Z_k$-equivariant Bauer-Furuta invariant}\label{BF-section}
For a more general $G$ action on $N$ than that of Theorem \ref{firstth}, we can compute  $\overline{BF}^{G}_{\bar{M}_k,\bar{\frak{s}}}$.
\begin{thm}\label{BF-thm}
Let $M$ and $N$ be  smooth closed oriented connected 4-manifolds satisfying
$b_2^+(M)> 1$ and $b_2^+(N)=0$, and $\bar{M}_k$ for any $k\geq 2$ be the connected sum
$M\#\cdots \#M\# N$ where there are $k$ summands of $M$.

Suppose that a finite group $G$ with $|G|=k$ acts effectively on $N$ in a smooth
orientation-preserving way, and that $N$ admits a Riemannian metric of positive scalar curvature invariant under the $G$-action and a $G$-equivariant Spin$^c$ structure $\frak{s}_N$ with $c_1^2(\frak{s}_N)=-b_2(N)$.

Define a $G$-action on $\bar{M}_k$ induced from that of $N$
permuting $k$ summands of $M$ glued along
a free orbit in $N$, and let $\bar{\frak{s}}$ be the
Spin$^c$ structure on $\bar{M}_k$ obtained by gluing $\frak{s}_N$ and a Spin$^c$
structure $\frak{s}$ of $M$.

Then for any $G$-action on $\bar{\frak{s}}$ covering the above
$G$-action on $\bar{M}_k$,
$$\overline{BF}^{G}_{\bar{M}_k,\bar{\frak{s}}}=BF_{M,\frak{s}}\wedge \overline{BF}^{G}_{N,\frak{s}_N},$$ and when $b_1(N)^{G}=0$,
$$\overline{BF}^{G}_{\bar{M}_k,\bar{\frak{s}}}=BF_{M,\frak{s}}.$$
If $BF_{M,\frak{s}}$ is nontrivial, so is $\overline{BF}^{G}_{\bar{M}_k,\bar{\frak{s}}}$.
\end{thm}
\begin{proof}
Let $\tilde{M}_k=N\cup \amalg_{i=1}^k(M\cup S^4)$ be the disjoint union of $N$ and $k$-copies of $M\cup S^4$,
and endow it with a Spin$^c$ structure $\tilde{\frak{s}}$ which is $\frak{s}_N$ on $N$,
$\frak{s}$ on each $M$, and the trivial Spin$^c$ structure $\frak{s}_0$ on each $S^4$. Then  $(\tilde{M}_k,\tilde{\frak{s}})$ has an obvious $G$-action induced from the $G$-action on $\bar{\frak{s}}$ in a unique way up
to homotopy. (Here $G$ acts on $\amalg_{i=1}^kS^4$ by the obvious  permutation, and on its Spin$^c$ structure as induced from the action on $\bar{\frak{s}}$ over the cylindrical gluing regions.)

Just as the ordinary monopole maps shown in \cite{bau}, the stable cohomotopy class of the disjoint union of $G$-monopole maps is equal to the smash product $\wedge $ of those, and hence
\begin{eqnarray*}
\overline{BF}^{G}_{\tilde{M}_k,\tilde{\frak{s}}}&=& \overline{BF}^{G}_{\amalg_{i=1}^k(M\cup S^4),\amalg_{i=1}^k(\frak{s}\amalg \frak{s}_0)}\wedge \overline{BF}^{G}_{N,\frak{s}_N}\\ &=& BF_{M\cup S^4,\frak{s}\amalg \frak{s}_0}\wedge \overline{BF}^{G}_{N,\frak{s}_N}\\ &=& BF_{M,\frak{s}}\wedge BF_{S^4,\frak{s}_0}\wedge \overline{BF}^{G}_{N,\frak{s}_N}\\ &=& BF_{M,\frak{s}}\wedge  \overline{BF}^{G}_{N,\frak{s}_N},
\end{eqnarray*}
where we used the fact that $BF_{S^4,\frak{s}_0}$ is just $[id]\in \pi^0_{S^1}(\textrm{pt})\cong \Bbb Z$, which was shown in \cite{bau}.

A surgery following S. Bauer \cite{bau} turns $\tilde{M}_k$ into the union of $\bar{M}_k$ and $k$-copies of $S^4\amalg S^4$. In the notations of \cite{bau}, for $X=X_1\cup X_2\cup X_3=\tilde{M}_k$, we take $$X_1=N=(N-\amalg_{i=1}^kD^4)\cup(\amalg_{i=1}^kD^4),$$ $$X_2=\amalg_{i=1}^kM=(\amalg_{i=1}^kD^4)\cup (\amalg_{i=1}^k(M-D^4)),$$ $$X_3=\amalg_{i=1}^kS^4=(\amalg_{i=1}^kD^4)\cup(\amalg_{i=1}^kD^4),$$ and
$$\tau=\left(                                                                                                      \begin{array}{ccc}                                                                                                        1 & 2 & 3 \\
2 & 3 & 1 \\                                                                                                      \end{array}                                                                                                    \right),$$ where $X_3$ is needed to make $\tau$  an even permutation so that ``the gluing map $V$" of Hilbert bundles along the necks is well-defined continuously. After interchanging the second half parts of $X_i$'s by $\tau$, we get  $$X^\tau=X_1^\tau\cup X_2^\tau\cup X_3^\tau=\bar{M}_k\cup(\amalg_{i=1}^kS^4)\cup (\amalg_{i=1}^kS^4)$$ as desired.\footnote{The gluing theorem 2.1 of \cite{bau} was stated when each $X_i$ is connected with one gluing neck, but the proof also works well without this assumption. For more details, readers are refereed to \cite{bau}.}

Most importantly, we perform the above surgery from $\tilde{M}_k$ to $\bar{M}_k\cup \amalg_{i=1}^{2k}S^4$ in a $G$-invariant way, and also ``the gluing map $V$" from the Hilbert bundles $\mathcal{E}^G,\mathcal{F}^G$  on $Pic^{G}(\tilde{M}_k)$ to the Hilbert bundles on $Pic^{G}(\bar{M}_k\cup \amalg_{i=1}^{2k}S^4)$ in a $G$-invariant way. The homotopy of the ordinary monopole map of $\tilde{M}_k$ shown in \cite{bau} can also be done in a $G$-invariant way. Then those $G$-monopole maps of $\tilde{M}_k$ and $\bar{M}_k\cup \amalg_{i=1}^{2k}S^4$ are conjugate via ``the gluing map $V$" up to $G$-invariant homotopy. Therefore their stable cohomotopy classes are equal so that
\begin{eqnarray*}
\overline{BF}^{G}_{\tilde{M}_k,\tilde{\frak{s}}}&=&\overline{BF}^{G}_{\bar{M}_k\cup \amalg_{i=1}^{2k}S^4,\bar{\frak{s}}\amalg \frak{s}_0}\\
&=& \overline{BF}^{G}_{\bar{M}_k,\bar{\frak{s}}}\wedge \overline{BF}^{G}_{\amalg_{i=1}^{2k}S^4,\frak{s}_0}\\
&=& \overline{BF}^{G}_{\bar{M}_k,\bar{\frak{s}}}\wedge  BF_{S^4\amalg S^4,\frak{s}_0}\\
&=& \overline{BF}^{G}_{\bar{M}_k,\bar{\frak{s}}}\wedge  BF_{S^4,\frak{s}_0}\wedge  BF_{S^4,\frak{s}_0}\\
&=& \overline{BF}^{G}_{\bar{M}_k,\bar{\frak{s}}},
\end{eqnarray*}
where we again used that $BF_{S^4,\frak{s}_0}=[id]$.

Therefore we obtained
$$\overline{BF}^{G}_{\bar{M}_k,\bar{\frak{s}}}=BF_{M,\frak{s}}\wedge \overline{BF}^{G}_{N,\frak{s}_N},$$ and it gets equal to $BF_{M,\frak{s}}$ in case of $b_1(N)^{G}=0$ by the following lemma :
\begin{lem}
If $b_1(N)^{G}=0$, then $\overline{BF}^{G}_{N,\frak{s}_N}$ is the class of the identity map $$[id]\in \pi^0_{S^1}(\textrm{pt})\cong \Bbb Z.$$
\end{lem}
\begin{proof}
The method of proof is basically the same as the ordinary Bauer-Furuta invariant in \cite{bau}.

First, we need to show that the $G$-index of the Spin$^c$ Dirac operator is zero. Take a $G$-invariant metric of positive scalar curvature on $N$.
Using the homotopy invariance of a $G$-index, we compute the index at a $G$-invariant connection $A_0$ whose curvature 2-form is harmonic and hence anti-self-dual.

Applying the Weitzenb\"ock formula with the fact that the scalar curvature of $N$ is positive, and the curvature 2-form is anti-self-dual, we get zero kernel.
Now then from the vanishing of the ordinary index given by $(c_1^2-\tau(N))/8$,   the cokernel must be also zero. In particular, we have vanishing of $G$-invariant kernel and cokernel, implying that the $G$-index is zero.

Then along with $b_1(N)^{G}=b_2^+(N)^{G}=0$, we conclude that $\overline{BF}^{G}_{N,\frak{s}_N}$ belongs to $\pi^0_{S^1}(\textrm{pt})$ which is isomorphic to $\pi^0_{st}(\textrm{pt})=\Bbb Z$ by the isomorphism induced by restriction to the $S^1$-fixed point set on which  the $G$-monopole map is just the linear isomorphism :
$$L^2_{m+1}(\Lambda^1N)^{G}\times H^0(N)^{G} \rightarrow L^2_{m}(\Lambda^2_+N)^{G}\times L^2_{m}(\Lambda^0N)^{G}\times H^1(N)^{G}$$
$$(a,c)\mapsto (d^+a,d^*a+c,a^{harm}),$$
because it has no kernel and cokernel. This completes the proof.
\end{proof}

Now let's consider the case of $b_1(N)^{G}\geq 1$.
Again the $G$-index bundle of the Spin$^c$ Dirac operator over $Pic^{G}(N)$ is zero so that $\overline{BF}^G_{N,\frak{s}_N}$ belongs to $\pi^0_{S^1}(T^{b_1(N)^{G}})$.

Following \cite{IL}, we consider the restriction map $$\sigma : \pi^0_{S^1}(T^{b_1(N)^{G}})\rightarrow \pi^0_{S^1}(\textrm{pt})$$  to the fiber over a point in $Pic^{G}(N)$. By the same method as the above lemma, $\sigma(\overline{BF}^G_{N,\frak{s}_N})$ is just the identity map.  Then the restriction of $\overline{BF}^G_{\bar{M}_k,\bar{\frak{s}}}=\overline{BF}^G_{\tilde{M}_k,\tilde{\frak{s}}}$ to $$Pic^{G}(\amalg_{i=1}^k(M\cup S^4))\times \{\textrm{pt}\} \subset Pic^{G}(\tilde{M}_k)=Pic^{G}(\bar{M}_k)$$ is given by $$BF_{M,\frak{s}}\wedge \sigma(\overline{BF}^G_{N,\frak{s}_N})=BF_{M,\frak{s}}.$$ It is obvious that  $\overline{BF}^{G}_{\bar{M}_k,\bar{\frak{s}}}$ is nontrivial, when $\sigma(\overline{BF}^G_{\bar{M}_k,\bar{\frak{s}}})$ is nontrivial, which completes the proof.
\end{proof}

\section{Examples of $(N,\frak{s}_N)$ of Theorem \ref{firstth}}

In this section,  $G, H$ and $K$ denote compact Lie groups. Let's  recall some elementary facts on equivariant principal bundles.
\begin{defn}
A principal $G$ bundle $\pi : P \rightarrow M$ is said to be $K$-equivariant if $K$ acts left on
both $P$ and $M$ in such a way that

(1) $\pi$ is $K$-equivariant :
$$\pi(k\cdot p) = k\cdot\pi(p)$$ for all $k\in K$ and $p\in P$,

(2) the left action of $K$ commutes with the right action of $G$ :
$$k\cdot(p\cdot g) = (k\cdot p)\cdot g$$ for all $k\in K, p\in P$, and $g\in G$.
\end{defn}
If $H$ is a normal subgroup of $G$, then one can define a principal $G/H$ bundle $P/H$ by taking the fiberwise quotient of $P$ by $H$. Moreover if $P$ is $K$-equivariant under a left $K$ action, then there exists the induced $K$ action on $P/H$ so that $P/H$ is $K$-equivariant.

\begin{lem}\label{jacob}
Let $P$ and  $\tilde{P}$ be a principal $G$ and $\tilde{G}$ bundle
respectively over a smooth manifold $M$ such that $\tilde{P}$
double-covers $P$ fiberwisely. For a normal subgroup $H$
containing $\Bbb Z_2$ in both $\tilde{G}$ and $S^1$ where the
quotient of $\tilde{G}$ by that $\Bbb Z_2$ gives $G$, let
$$\tilde{P}\otimes_{H}S^1:=(\tilde{P}\times_M (M\times S^1))/ H$$
be the quotient of the fiber product of $\tilde{P}$ and the
trivial $S^1$ bundle $M\times S^1$ by $H$, where the right $H$
action is given by $$(p,(x,e^{i\vartheta}))\cdot h=(p\cdot h,
(x,e^{i\vartheta}h^{-1})).$$

Suppose that $M$ and $P$ admit a smooth $S^1$ action such that $P$
is $S^1$-equivariant.  Then a principal $\tilde{G}\otimes_{H}S^1$
bundle $\tilde{P}\otimes_{H}S^1$ is also $S^1$-equivariant by
lifting the action on $P$. In particular, any smooth $S^1$-action
on a smooth spin manifold lifts to its trivial Spin$^c$ bundle so
that the Spin$^c$ structure is $S^1$-equivariant.
\end{lem}
\begin{proof}
Any left $S^1$ action on $P$ can be lifted to $\tilde{P}$ uniquely
at least locally commuting with the right $\tilde{G}$ action. If
the monodromy is trivial for any orbit, then the $S^1$ action can
be globally well-defined on $\tilde{P}$, and hence on
$\tilde{P}\otimes_{H}S^1$, where the $S^1$ action on the latter
$S^1$ fiber can be any left action, e.g. the trivial action,
commuting with the right $S^1$ action.

If the monodromy is not trivial, it has to be $\Bbb Z_2$ for any
orbit,  because the orbit space is connected. In that case, we
need the trivial $S^1$ bundle $M\times S^1$ with an ``ill-defined"
$S^1$ action with monodromy $\Bbb Z_2$ defined as follows.

First consider the double covering map from $M\times S^1$ to
itself defined by  $(x,z)\mapsto (x,z^2)$. Equip the downstairs
$M\times S^1$ with the left $S^1$ action which acts on the base as
given and on the fiber $S^1$ by the  multiplication as complex
numbers. Then this downstairs action can be locally lifted to the
upstairs commuting with the right $S^1$ action. Most importantly,
it has $\Bbb Z_2$ monodromy as desired. Explicitly,
$e^{i\vartheta}$ for $\vartheta\in[0,2\pi)$ acts on the fiber
$S^1$ by the  multiplication of $e^{i\frac{\vartheta}{2}}$.
Combining this with the local action on $\tilde{P}$, we get a
well-defined $S^1$ action on $\tilde{P}\otimes_{H}S^1$, because
two $\Bbb Z_2$ monodromies are cancelled each other.

Once the $S^1$ action on $\tilde{P}\otimes_{H}S^1$ is globally well-defined, it commutes with the right $\tilde{G}\otimes_{H}S^1$  action, because the local $S^1$ action on $\tilde{P}\times S^1$ commuted with the right $\tilde{G}\times S^1$ action.

If $S^1$ acts on a smooth manifold, the orthonormal frame bundle is always $S^1$-equivariant under the action. Then by the above result  any $S^1$ action on a smooth spin manifold lifts to  the trivial Spin$^c$ bundle which is $(\textrm{spin bundle})\otimes_{\Bbb Z_2}S^1$.
\end{proof}

\begin{lem}\label{joseph}
Let $P$ be a flat principal $G$ bundle over a smooth manifold $M$ with a smooth $S^1$ action. Suppose that the action can be lifted to the universal cover $\tilde{M}$ of $M$. Then it can be also   lifted to $P$ so that $P$ is $S^1$-equivariant.
\end{lem}
\begin{proof}
For the covering map $\pi: \tilde{M}\rightarrow M$, the pull-back
bundle $\pi^*P$ is the trivial bundle $\tilde{M}\times G$.  By
letting $S^1$ act on the fiber $G$ trivially, $\pi^*P$ can be made
$S^1$-equivariant. For the deck transformation group $\pi_1(M)$,
$P$ is gotten by an element of  $\textrm{Hom}(\pi_1(M),G)$. Any
deck transformation acts on each fiber $G$  as the left
multiplication of a constant in $G$ so that it commutes with not
only the right $G$ action but also the left $S^1$ action which is
trivial on the fiber $G$. Therefore the $S^1$ action on $\pi^*P$
projects down to an $S^1$ action on $P$. To see whether this $S^1$
action commutes with the right $G$ action, it's enough to check
for the local $S^1$ action, which can be seen upstairs on
$\pi^*P$.
\end{proof}

\begin{lem}
On a smooth closed oriented 4-manifold $N$ with $b_2^+(N)=0$, any  Spin$^c$ structure $\frak{s}$ satisfies $$c_1^2(\frak{s})\leq -b_2(N),$$ and the choice of a Spin$^c$ structure $\frak{s}_N$ satisfying $c_1^2(\frak{s}_N)=-b_2(N)$ is always possible.
\end{lem}
\begin{proof}
If $b_2(N)=0$, it is obvious. The case of $b_2(N)>0$ can be seen
as follows.  Using Donaldson's theorem \cite{donal1,donal2}, we
diagonalize the intersection form $Q_N$ on $H^2(N;\Bbb
Z)/\textrm{torsion}$ over $\Bbb Z$ with a basis
$\{\alpha_1,\cdots,\alpha_{b_2(N)}\}$ satisfying
$Q_N(\alpha_i,\alpha_i)=-1$ for all $i$. Then for any Spin$^c$
structure $\frak{s}$, the rational part of $c_1(\frak{s})$ should
be of the form $$\sum_{i=1}^{b_2(N)}a_i\alpha_i$$ where each
$a_i\equiv 1$ mod 2, because $$Q_N(c_1(\frak{s}),\alpha)\equiv
Q_N(\alpha,\alpha)\ \ \ \ \textrm{mod}\ 2$$ for any $\alpha\in
H^2(N;\Bbb Z)$. Consequently $|a_i|\geq 1$ for all $i$ which means
$$c_1^2(\frak{s})=\sum_{i=1}^{b_2(N)}-a_i^2\leq -b_2(N),$$ and we
can get a Spin$^c$ structure $\frak{s}_N$ with
$$c_1(\frak{s}_N)\equiv\sum_i \alpha_i\ \ \ \ \textrm{modulo
torsion}$$ by tensoring any $\frak{s}$ with a line bundle $L$
satisfying  $$2c_1(L)+c_1(\frak{s})\equiv\sum_i \alpha_i\ \ \ \
\textrm{modulo torsion},$$ completing the proof.
\end{proof}

\begin{thm}\label{Nexam}
Let $X$ be one of $$S^4,\ \  \overline{\Bbb CP}_2,\ \ S^1\times
(L_1\#\cdots\#L_n),\ \ \textrm{and}\ \ \widehat{S^1\times L}$$
where  each $L_i$ and $L$ are quotients of $S^3$ by free actions
of finite groups, and $\widehat{S^1\times L}$ is the manifold
obtained from the surgery on $S^1\times L$ along an $S^1\times
\{pt\}$.

Then for any integer $l\geq 0$ and any smooth
closed oriented 4-manifold $Z$ with $b_2^+(Z)=0$  admitting a
metric of positive scalar curvature, $$X\ \#\ kl Z$$  satisfies the properties of $N$ in
Theorem \ref{firstth}, where the Spin$^c$ structure of $X \# kl Z$ is given by gluing any Spin$^c$ structure $\frak{s}_X$ on $X$ and any Spin$^c$ structure $\frak{s}_Z$ on $Z$ satisfying $c_1^2(\frak{s}_X)=-b_2(X)$ and $c_1^2(\frak{s}_Z)=-b_2(Z)$ respectively.
\end{thm}
\begin{proof}
First, we will define  $\Bbb Z_k$ actions  preserving a metric of positive scalar curvature.
In fact, our actions on $X$ will be induced from such $S^1$ actions.

For $X=S^4$, one can take a $\Bbb Z_k$-action coming from a
nontrivial action of $S^1\subset SO(5)$ preserving a round metric. In this case, one can choose a free action or an action with fixed points also.

If $X=\overline{\Bbb CP}_2$, then one can use the following actions for some integers $m_1, m_2$ :
\begin{eqnarray}\label{exam}
j\cdot [z_0,z_1,z_2]=[z_0,e^{\frac{2jm_1}{k}\pi i}z_1,e^{\frac{2jm_2}{k}\pi i}z_2]
\end{eqnarray}
for $j\in \Bbb Z_k$, which preserve the Fubini-Study metric and has at least 3 fixed points $[1,0,0], [0,1,0], [0,0,1]$.

Before considering the next example, recall that every finite
group acting freely on $S^3$ is in fact conjugate to a subgroup of
$SO(4)$, and hence its quotient 3-manifold admits a metric of
constant positive curvature. This follows from the well-known
result of G. Perelman. (See \cite{morgan-tian1, morgan-tian2}.)

In $S^1\times (L_1\#\cdots\#L_n)$, the action is defined as a rotation along the $S^1$-factor, which is
obviously free and preserves a product metric. By endowing $L_1\#\cdots\#L_n$  with a metric of positive scalar curvature via the Gromov-Lawson surgery \cite{GL}, $S^1\times (L_1\#\cdots\#L_n)$ has a desired metric.

Finally the above-mentioned $S^1$ action on
$S^1\times L$ can be naturally extended to $\widehat{S^1\times
L}$, and moreover the Gromov-Lawson surgery \cite{GL} on
$S^1\times\{pt\}$ produces an $S^1$-invariant metric of positive
scalar curvature. Its fixed point set is $\{0\}\times S^2$ in the attached $D^2\times S^2$.

Now $X\# kl Z$ has an obvious $\Bbb Z_k$-action induced
from that of $X$ and a $\Bbb Z_k$-invariant metric which has positive
scalar curvature again by the Gromov-Lawson surgery.

It remains to prove that the above $\Bbb Z_k$-action on $X \# kl Z$ can be lifted to the Spin$^c$ structure obtained by gluing the above $\frak{s}_X$  and  $\frak{s}_Z$.
For this, we will only prove that any such $\frak{s}_X$  is $\Bbb Z_k$-equivariant. Then one can glue $k$ copies of $lZ$ in an obvious $\Bbb Z_k$-equivariant way. Recalling that the $\Bbb Z_k$ action on $X$ actually comes from an $S^1$ action, we will actually show the $S^1$-equivariance of  $\frak{s}_X$ on $X$.

On $S^4$,  the unique Spin$^c$ structure is trivial.  Any smooth $S^1$ action on $S^4$ which is spin can be lifted its trivial Spin$^c$ structure by Lemma \ref{jacob}.

Any smooth $S^1$ action on $\overline{\Bbb CP}_2$ is uniquely lifted to its orthonormal frame bundle $F$, and any  Spin$^c$ structure on $\overline{\Bbb CP}_2$
satisfying $c_1^2=-1$ is the double cover $P_1$ and $P_2$  of $F\oplus P$ and $F\oplus P^*$ respectively, where $P$ is the principal $S^1$ bundle over $\overline{\Bbb CP}_2$ with $c_1(P)=[H]$ and $P^*$ is its dual. Note that there is a base-preserving diffeomorphism between $P$ and $P^*$ whose total space is $S^5$. Obviously the action
(\ref{exam}) is extended to $S^5\subset \Bbb C^3$ commuting with the principal $S^1$ action of the Hopf fibration. By Lemma \ref{jacob} the $S^1$-action can be lifted to  $P_i\otimes_{S^1} S^1$ in an $S^1$-equivariant way, which is isomorphic to $P_i$ for $i=1,2$.

In case of $S^1\times (L_1\#\cdots\#L_n)$, any Spin$^c$ structure is
the pull-back from $L_1\#\cdots\#L_n$, and satisfies $c_1^2=0=-b_2$. Because the
tangent bundle is trivial, a free $S^1$-action is obviously
defined on its trivial spin bundle. Then the action can be
obviously extended to any Spin$^c$ structure, because it is
pulled-back from $L_1\#\cdots\#L_n$.

\begin{lem}
$\widehat{S^1\times L}$ is a rational homology 4-sphere, and  $$H^2(\widehat{S^1\times L};\Bbb Z)=H_1(L;\Bbb Z).$$ Its universal cover is $(|\pi_1(L)|-1)S^2\times S^2$ where $0(S^2\times S^2)$ means $S^4$.
\end{lem}
\begin{proof}
Since the Euler characteristic is easily computed to be 2 from the
surgery description,  and $b_1(\widehat{S^1\times L})=b_1(L)=0$,
it follows that $\widehat{S^1\times L}$ is a rational homology
4-sphere.

By the universal coefficient theorem,
\begin{eqnarray*}
H^2(\widehat{S^1\times L};\Bbb Z)
&=&\textrm{Hom}(H_2(\widehat{S^1\times L};\Bbb Z),\Bbb Z)\oplus \textrm{Ext}(H_1(\widehat{S^1\times L};\Bbb
Z),\Bbb Z)\\ &=& H_1(\widehat{S^1\times L};\Bbb Z)\\ &=& H_1(L;\Bbb Z).
\end{eqnarray*}

The universal cover is equal to the manifold obtained from $S^1\times S^3$ by performing surgery along  $S^1\times \{ |\pi_1(L)|\ \textrm{points in } S^3\}$, and hence it must be $(|\pi_1(L)|-1)S^2\times S^2$.
\end{proof}

By the above lemma, there are $|H_1(L;\Bbb Z)|$ Spin$^c$
structures on $\widehat{S^1\times L}$, all of which are torsion to satisfy
$c_1^2=0=-b_2(\widehat{S^1\times L})$. Since any $S^1$ bundle on  $\widehat{S^1\times L}$ is flat, and the $S^1$-action on  $\widehat{S^1\times L}$ can be obviously lifted to its universal cover, Lemma \ref{joseph} says that any $S^1$ bundle is $S^1$-equivariant under the $S^1$ action.

By the construction, $\widehat{S^1\times L}$ is spin, and hence the trivial Spin$^c$ bundle is $S^1$-equivariant by Lemma \ref{jacob}.
Any other Spin$^c$ structure is given by the tensor product over $S^1$ of the trivial Spin$^c$ bundle and an $S^1$ bundle, both of which are $S^1$-equivariant bundles. Therefore any Spin$^c$ bundle of $\widehat{S^1\times L}$ is  $S^1$-equivariant.

This completes all the proof.
\end{proof}


\section{Exotic group actions}
In this section, let's abbreviate $SW^G_{M,\mathfrak{s}}(U^d)$ for $d$ equal to the expected dimension of its $G$-monopole moduli space simply by $SW^G_{M,\mathfrak{s}}$ by abuse of notation. Generalizing the Seiberg-Witten polynomial $SW_M$ of $M$, let's define the $G$-monopole polynomial of $M$  as
$$SW^G_M:=\sum_{\mathfrak{s}}SW^G_{M,\mathfrak{s}}PD(c_1(\mathfrak{s}))\in \Bbb Z[H_2(M;\Bbb Z)^G],$$ where the summation is over the set of all $G$-equivariant
Spin$^c$ structures. We also use the term {\it mod 2 basic class} to mean the first Chern class of a Spin$^c$ structure with nonzero mod 2 Seiberg-Witten invariant.

Following \cite{GoS}, we say that a simply connected 4-manifold {\it dissolves} if it is diffeomorphic either to $$n{\Bbb CP}_2\#m\overline{\Bbb CP}_2\ \ \ \textrm{or to}\ \ \ \pm(n(S^2\times S^2)\#mK3)$$ for some $n,m\geq 0$ according to its homeomorphism type.

\begin{thm}\label{upgrade1}
Let $M$  be a smooth closed oriented connected 4-manifold and $\{M_i|i\in \frak I\}$ be a family of smooth 4-manifolds such that every $M_i$ is homeomorphic to $M$ and the numbers of mod 2 basic classes of $M_i$'s are all mutually different, but each $M_i\#l_i(S^2\times S^2)$ is diffeomorphic to $M\#l_i(S^2\times S^2)$ for an integer $l_i\geq 1$.

If $l_{max}:=\sup_{i\in\frak I}l_{i}<\infty$, then for any integers $k\geq 2$ and $l\geq l_{max}+1$, $$klM\#(l-1)(S^2\times S^2)$$
admits an $\frak I$-family of topologically equivalent but smoothly
distinct non-free actions of $\Bbb Z_k\oplus H$ where $H$ is any
group of order $l$ acting freely on $S^3$.
\end{thm}
\begin{proof}
Think of $klM\#(l-1)(S^2\times S^2)$ as $$klM_i\#(l-1)(S^2\times
S^2),$$ and our $H$ action is defined as the deck transformation
map of the $l$-fold covering map onto
$$\bar{M}_{i,k}:=kM_i\#\widehat{S^1\times L}$$ where $\widehat{S^1\times L}$ for $L=S^3/H$ is defined as in Theorem \ref{Nexam}. To
define a $\Bbb Z_k$-action, note that $\bar{M}_{i,k}$ has a $\Bbb Z_k$-action coming from the $\Bbb Z_k$-action of $\widehat{S^1\times L}$ defined in Theorem \ref{Nexam}, which is basically a
rotation along the $S^1$-direction. This $\Bbb Z_k$ action is
obviously lifted to the above $l$-fold cover, and  it commutes
with the above defined $H$ action. Thus we have an $\frak I$-family
of $\Bbb Z_k\oplus H$ actions on $klM\#(l-1)(S^2\times S^2)$,
which are all topologically equivalent by using the homeomorphism
between each $M_i$ and $M$.

Recall from Theorem \ref{Nexam} that all the Spin$^c$ structures
on $\widehat{S^1\times L}$ are $\Bbb Z_k$-equivariant and satisfy
$c_1^2=-b_2(\widehat{S^1\times L})=0$. By Theorem \ref{myLord} and
the fact that $b_1(\widehat{S^1\times L})=0$, for any Spin$^c$
structure $\frak{s}_i$ on $M_i$,
$$SW^{\Bbb Z_k}_{\bar{M}_{i,k},\bar{\frak{s}}_i}\equiv SW_{M_i,\frak{s}_i}\ \ \ \textrm{mod}\ 2,$$
and hence $$SW^{\Bbb
Z_k}_{\bar{M}_{i,k}}\equiv SW_{M_i}\sum_{[\alpha]\in
H^2(\widehat{S^1\times L}; \Bbb Z)}[\alpha]$$ modulo 2. Therefore $SW^{\Bbb
Z_k}_{\bar{M}_{i,k}}$ mod 2 for all $i$ have mutually different numbers of monomials, and hence the $\frak I$-family of $\Bbb Z_k\oplus H$ actions on $klM\#(l-1)(S^2\times S^2)$ cannot
be smoothly equivalent, completing the proof.
\end{proof}


\begin{cor}\label{upgrade2}
Let $H$ be a finite group of order $l\geq 2$ acting freely on $S^3$. For any $k\geq 2$, there exists an infinite
family of topologically equivalent but smoothly distinct non-free
actions of $\Bbb Z_k\oplus H$ on $$(klm+l-1)(S^2\times S^2),$$
$$(kl(n-1)+l-1)(S^2\times S^2)\#klnK3,$$
$$(kl(2n'-1)+l-1)\Bbb CP_2\# (kl(10n'+m'-1)+l-1)\overline{\Bbb CP}_2$$
 for infinitely many $m$, and any $m'\geq 1$, $n, n'\geq 2$. 
\end{cor}
\begin{proof}
By the result of B. Hanke, D. Kotschick, and J. Wehrheim \cite{Kot}, $m(S^2\times S^2)$ for infinitely many $m$ has the property of $M$ in the above theorem with each $l_i=1$ and $|\frak I|=\infty$. 
The different smooth structures of their examples are constructed by fiber-summing  a logarithmic transform of $E(2n)$ and a certain symplectic 4-manifold along a symplectically embedded torus, and different numbers of mod 2 basic classes are due to those different logarithmic transformations. Indeed the Seiberg-Witten polynomial of the multiplicity $r$ logarithmic transform of $E(2n)$ is given by 
$$([T]^r-[T]^{-r})^{2n-2}([T]^{r-1}+[T]^{r-3}+\cdots +[T]^{1-r})$$ whose number of terms with coefficients mod 2 can be made arbitrarily large by taking $r$ sufficiently large, and the fiber sum with the other symplectic 4-manifold is performed on a fiber in an  $N(2)$ disjoint from the Gompf nucleus $N(2n)$ where the log transform is performed so that all these mod 2 basic classes survive the fiber-summing by the gluing formula of C. Taubes \cite{taubes}.
Therefore $(klm+l-1)(S^2\times S^2)$ has desired actions by the above theorem.

For the second example, we use a well-known fact that $E(n)$ for $n\geq 2$ also has the above properties of $M$ in the above theorem with each $l_i=1$, where its exotica $M_i$'s are $E(n)_K$ for a knot $K\subset S^3$
by the Fintushel-Stern knot surgery. Recall the theorem by S. Akbulut \cite{Akb} and D. Auckly \cite{Auc} which says that for any smooth closed simply-connected $X$ with an embedded torus $T$ such that $T\cdot T=0$ and $\pi_1(X-T)=0$, a knot-surgered manifold $X_K$ along $T$ via a knot $K$ satisfies $$X_K\#(S^2\times S^2)=X\#(S^2\times S^2).$$
And from the formula
\begin{eqnarray*}
SW_{E(n)_K}&=& \Delta_K([T]^2)([T]-[T]^{-1})^{n-2}
\end{eqnarray*}
where $\Delta_K$ is the symmetrized Alexander polynomial of $K$,
one can easily see that the number of mod 2 basic classes of
$E(n)_K$ can be made arbitrarily large by choosing $K$
appropriately. (For example, take $K$ with
$$\Delta_K(t)=1+\sum_{j=1}^{2d}(-1)^j(t^{jn}+t^{-jn})$$ for
sufficiently large $d$.) Therefore $$klE(2n)\# (l-1)(S^2\times
S^2)=klnK3\#(kl(n-1)+l-1)S^2\times S^2$$ has desired actions, where
we used the fact that $S\#(S^2\times S^2)$ dissolves for any
smooth closed simply-connected elliptic surface $S$ by the work of
R. Mandelbaum \cite{ma} and R. Gompf \cite{Go3}.

For the third example, one can take $M$ to be
$E(n')\#m'\overline{\Bbb CP}_2$ for $n'\geq 2, m'\geq 1$, where its exotica $M_i$'s are $E(n')_K\#m'\overline{\Bbb CP}_2$ for a knot $K\subset S^3$, because
\begin{eqnarray*}
SW_{E(n')_K\#m'\overline{\Bbb CP}_2}&=& SW_{(E(n')\#m'\overline{\Bbb CP}_2)_K}\\ &=& \Delta_K([T]^2)([T]-[T]^{-1})^{n'-2}\prod_{i=1}^{m'} ([E_{i}]+[E_{i}]^{-1}),
\end{eqnarray*}
where $E_i$'s denote the exceptional divisors, and we used the fact that $E(n')$ is of simple type. Since
$E(n')\#\overline{\Bbb CP}_2$ for any $n'$ is non-spin,
$$kl(E(n')\#m'\overline{\Bbb CP}_2)\#(l-1)(S^2\times S^2)=kl(E(n')\#m'\overline{\Bbb CP}_2)\#(l-1)(\Bbb CP_2\# \overline{\Bbb CP}_2),$$
and it dissolves into the connected sum of $\Bbb CP_2$'s and
$\overline{\Bbb CP}_2$'s, using the dissolution (\cite{ma,Go3}) of
$E(n')\#\Bbb CP_2$ into $2n'\Bbb CP_2\#(10n'-1)\overline{\Bbb
CP}_2$.
\end{proof}

\begin{rmk}
For other combinations of $K3$ surfaces and $S^2\times S^2$'s in the above corollary, one can use B. Hanke, D. Kotschick, and J. Wehrheim's other examples in \cite{Kot}.
One can also construct many other such examples of $M$ with infinitely many exotica which become diffeomorphic after one stabilization by using the knot surgery.

Any finite group acting freely on $S^3$ is in fact a subgroup of
$SO(4)$ by the well-known result of G. Perelman
(\cite{morgan-tian1, morgan-tian2}), and Theorem \ref{upgrade1}
and Corollary \ref{upgrade2} can be generalized a little further.
(See \cite{sung3}.)
\end{rmk}

\bigskip

\noindent{\bf Acknowledgement.} The author would like to express sincere thanks to Prof. Ki-Heon Yun for helpful discussions and supports.

\end{document}